\documentclass[12pt,usenames,dvipsnames]{amsart}

\usepackage{amsmath,amssymb,amsthm,array,bm,calrsfs,comment,enumitem,graphicx,hyperref,latexsym,multirow,tikz-cd,mathtools}
\usepackage{mathpazo}
\DeclareMathAlphabet{\pazocal}{OMS}{zplm}{m}{n}
\usepackage{float}

\usepackage[margin=1in]{geometry}

\usepackage[cmtip, all]{xy}

\newcommand{\C}{\mathbb{C}}
\newcommand{\F}{\mathbb{F}}
\newcommand{\G}{\mathbb{G}}
\newcommand{\N}{\mathbb{N}}
\renewcommand{\P}{\mathbb{P}}
\newcommand{\Q}{\mathbb{Q}}
\newcommand{\R}{\mathbb{R}}
\newcommand{\Z}{\mathbb{Z}}

\newcommand{\pB}{\pazocal{B}}
\newcommand{\pH}{\pazocal{H}}
\newcommand{\pL}{\pazocal{L}}
\newcommand{\pQ}{\pazocal{Q}}
\newcommand{\pR}{\pazocal{R}}
\newcommand{\pZ}{\pazocal{Z}}

\newcommand{\sL}{\mathsf{L}}
\newcommand{\sQ}{\mathsf{Q}}
\newcommand{\sU}{\mathsf{U}}

\newcommand{\sa}{\mathsf{a}}
\newcommand{\sfp}{\mathsf{p}}
\newcommand{\sq}{\mathsf{q}}
\newcommand{\su}{\mathsf{u}}
\newcommand{\sv}{\mathsf{v}}
\newcommand{\sw}{\mathsf{w}}

\newcommand{\cl}{\mathsf{cl}}
\newcommand{\conv}{\mathsf{conv}}
\newcommand{\CCt}{\C(\!(t)\!)}
\newcommand{\ETrop}{\mathfrak{Trop}}
\newcommand{\Fl}{\mathsf{Fl}}
\newcommand{\Gr}{\mathsf{Gr}}
\newcommand{\init}{\mathsf{in}}
\newcommand{\PGL}{\mathsf{PGL}}
\newcommand{\pr}{\mathsf{pr}}

\newcommand{\sing}{\mathsf{sing}}
\newcommand{\sgn}{\mathsf{sgn}}
\newcommand{\smgp}{\mathsf{smgp}}
\newcommand{\Spec}{\mathsf{Spec}}
\newcommand{\TGr}{\mathsf{TGr}}
\newcommand{\trop}{\mathsf{trop}}
\newcommand{\Trop}{\mathsf{Trop}}
\newcommand{\val}{\mathsf{val}}
\newcommand{\red}{\mathsf{red}}
\newcommand{\Zone}{\langle \mathbf{1}\rangle}

\newcommand{\bk}[2]{\langle #1, #2 \rangle}
\newcommand{\Leaf}[1]{\mathsf{L}(#1)}
\newcommand{\pe}[3]{#1\, +_{#2} \, #3}

\newcommand{\Sn}[1]{\mathfrak{S}_{#1}}

\newtheorem{theorem}{Theorem}[section]
\newtheorem{lemma}[theorem]{Lemma}
\newtheorem{proposition}[theorem]{Proposition}
\newtheorem{corollary}[theorem]{Corollary}

\newtheorem*{theorem*}{Theorem}

\theoremstyle{definition}
\newtheorem{question}[theorem]{Question}

\newtheorem{example}[theorem]{Example}

\theoremstyle{remark}
\newtheorem{remark}[theorem]{Remark}

\numberwithin{equation}{section}
\numberwithin{table}{section}
\numberwithin{figure}{section}

\title{Singular matroid realization spaces}

\date{\today}

\author{Daniel Corey}
\email{daniel.corey@unlv.edu}
\address{University of Nevada, Las Vegas, 
4505 S. Maryland Pkwy., 
Las Vegas, NV 89154}

\author{Dante Luber}
\email{luber@math.uni-frankfurt.de}
\address{Goethe Universit\"at, Frankfurt Am Main, Robert-Mayer-Str. 6-8 D-60325 Frankfurt am Main Germany.}


\begin{document}

\begin{abstract}
    We study smoothness of realization spaces of matroids for small rank and ground set. For $\C$-realizable matroids, when the rank is  $3$, we prove that the realization spaces are all smooth when the ground set has  $11$ or fewer elements, and there are singular realization spaces for $12$ and greater elements.  For rank $4$ and $9$ or fewer elements, we prove that these realization spaces are smooth. As an application, we prove that $\Gr^{\circ}(d,n;\C)$---the locus of the Grassmannian where all Pl\"ucker coordinates are nonzero---is not sch\"on for all but finitely many pairs $(d,n)$ such that $3\leq d \leq n-3$. 
\end{abstract}

\maketitle

\section{Introduction}\label{sec:intro}

Matroid realization spaces---algebraic varieties, or more generally schemes, that parameterize hyperplane arrangements realizing a fixed matroid---provide a rich source of singular spaces in algebraic geometry, at least at an abstract level. By Mn\"ev's universality theorem \cite{Mnev85, Mnev88}, for each singularity type (a notion made precise by Vakil in \cite{Vakil}) there is a rank $3$ matroid $\sQ$ such that this singularity type appears in the realization space of $\sQ$. This theorem serves as a template proving that many moduli spaces satisfy \textit{Murphy's law}, i.e., that every singularity type appears on that moduli space.

While the proofs of Mn\"{e}v's universality theorem are constructive (see \cite[Theorem~5.3]{Cartwright},  \cite[I.14]{Lafforgue2003}, or \cite{VakilLee}) the size of the ground set becomes large even for the simplest singularities. Thus, the primary goal of this paper is to find the smallest $n$ such that there is a rank $3$ matroid on $n$ elements whose realization space is singular over $\C$. 

In earlier work \cite{CoreyGrassmannians,CoreyLuber}, the authors prove that the realization spaces of $\C$--realizable rank $3$ matroids on $8$ or fewer elements are all smooth (which is closely related to earlier connectivity results in \cite{GalletSaini, NazirYoshinaga}).  Furthermore, they are irreducible except for the realization space of the M\"{o}bius--Kantor matroid.  This is the matroid associated to the unique arrangement of 8 points and 8 lines in $\P^{2}$ such that each point lies on exactly 3 lines, and each line contains exactly 3 points. We prove that the answer to problem posed in the previous paragraph is  $n=12$. Our main results are summarized below.

\begin{theorem}
\label{thm:intro-3-leq11-smooth}
    The realization spaces of $\C$-realizable rank $3$ matroids on 11 or fewer elements are smooth over $\C$. 
\end{theorem}

\begin{theorem}
\label{thm:intro-3-geq12-singular}
    For $n\geq 12$, there exists a rank $3$ matroid on $n$ elements whose realization space has nodal singularities over $\C$.     
\end{theorem}

\noindent 
We also initiate a systematic study of smoothness and irreducibility for rank $4$ matroids. 

\begin{theorem}\label{thm:intro-4-n-matroids-smooth}
    For $n\leq 9$, the realization spaces of $\C$-realizable rank $4$ matroids on $n$ elements are smooth.
\end{theorem}

\noindent There are 13 pairs $(d,n)$ for which we do not know if there are singular matroid realization spaces for $\C$-realizable matroids of rank $d$ on $n$ elements.  See \eqref{eq:openPairs}.

To prove Theorems \ref{thm:intro-3-leq11-smooth} and \ref{thm:intro-4-n-matroids-smooth}, we use a combination of theoretical results and computer computations in \texttt{OSCAR}. We prove structural results that allow us to deduce smoothness of realization spaces for some matroids by knowing smoothness of realization spaces of matroids on smaller ground sets. For example, the singularity types in a realization space are preserved under \textit{principal extension}, i.e., the act of freely adding an element to a flat of a matroid. See Proposition \ref{prop:principalExtension}, as well as  Proposition \ref{prop:2planes} for a stronger criterion. We use in a crucial way the database of matroids developed by \cite{MatsumotoMoriyamaImaiBremner}. There are $500\, 957$ (possibly nonrealizable)  matroids covered by these two theorems, and using our structural results, we need only verify smoothness directly for $50\,945$ of them, about $10\%$ of the total.

Our main application of Theorem \ref{thm:intro-3-geq12-singular} is to find singular initial degenerations of the Grassmannian. Recall that the Grassmannian $\Gr(d,n;\C)$ is the algebraic variety that parameterizes $d$-dimensional linear subspaces of $\C^{n}$. Via the Pl\"ucker embedding, $\Gr(d,n;\C)$ is realized as a closed subvariety of $\P^{\binom{n}{d}-1}$, and its homogeneous coordinates are called \textit{Pl\"ucker coordinates}. The open subvariety $\Gr^{\circ}(d,n;\C)$ defined by the nonvanishing of all Pl\"ucker coordinates is a widely studied object in tropical geometry and moduli theory. Its tropicalization $\TGr^{\circ}(d,n)$ is closely related to phylogenetic trees (for $n=2$), valuated matroids, and toric degenerations \cite{SpeyerSturmfels2004a}.  The space $\Gr^{\circ}(d,n;\C)$ is also connected, via the Gelfand--MacPherson correspondence, to the moduli space $\pR(d,n;\C)$ of $n$ hyperplanes in $\P_{\C}^{d-1}$ in linear general position up to the action of $\PGL_{d}(\C)$ (alternately, $\pR(d,n;\C)$ is the realization space of the uniform matroid).  

In \cite{CoreyGrassmannians,CoreyLuber}, the authors prove that the normalization of the Chow quotient of $\Gr(3,n;\C)$ by the diagonal torus of $\PGL_{n}(\C)$ is the log canonical model of $\pR(3,n;\C)$ for $n\in\{7,8\}$, and the $n=6$ case was earlier proved in \cite{Luxton}. These works fully resolve \cite[Conjecture~1.6]{KeelTevelev2006}. The key step in the proof is to show that $\Gr^{\circ}(3,n;\C)$ is sch\"on for $n\leq 8$. (A closed subvariety of an algebraic torus is \textit{sch\"on} if its initial degenerations are all smooth; we review these notions in \S\ref{sec:initial_degenerations}).  Proving that a variety is sch\"on is, in general, a challenging task. The Grassmannian case relies heavily on a result of the first author \cite[Theorem~1.1]{CoreyGrassmannians} which relates the initial degenerations of $\Gr^{\circ}(3,n;\C)$ to matroid strata of the Grassmannian. Because of this connection and Mn\"ev's universality theorem, one may suspect that $\Gr^{\circ}(d,n;\C)$ is, in general, not sch\"on. Using Theorem \ref{thm:intro-3-geq12-singular}, we confirm this suspicion in the following theorem. 

\begin{theorem}
\label{thm:not-shon-geq-12}
    Let $3\leq d \leq \frac{n}{2}$. The open Grassmannian $\Gr^{\circ}(d,n;\C)$ and moduli space $\pR(d,n)$ are not sch\"on for all but finitely many pairs $(d,n)$. 
\end{theorem}

\noindent
Since the initial degenerations of $\Gr^{\circ}(d,n;\C)$ and $\pR(d,n)$ differ by a torus factor, sch\"oness of $\Gr^{\circ}(d,n;\C)$ is equivalent to that of $\pR(d,n)$. The restriction of $d\leq \frac{n}{2}$ poses no loss of generality because of the duality isomorphism $\Gr^{\circ}(d,n) \cong \Gr^{\circ}(n-d,n)$, and the restriction $d\geq 3$ is necessary since $\Gr^{\circ}(d,n;\C)$ is sch\"on for $d=1$ and $2$ \cite{Tevelev}. There are 18 pairs $(d,n)$ for which we do not know if $\Gr^{\circ}(d,n)$ is sch\"on. These may be found in \eqref{eq:openPairs}.

\begin{question}\label{question:3-n range}
    For which pairs $(d,n)$ in \eqref{eq:openPairs} is $\Gr^{\circ}(d,n;\C)$ sch\"on? 
\end{question}

\noindent Resolving the above cases in the affirmative could be challenging, as the techniques used in prior work and this paper are already known to fail for $(d,n)=(3,9)$, see \cite[Example 8.2]{CoreyGrassmannians}.

\subsection*{Code}\label{code} We use \texttt{OSCAR} \cite{OSCAR-book,OSCAR} which runs using \texttt{julia} \cite{BezansonEdelmanKarpinskiShah}. Specifically, \texttt{OSCAR} is essential to the proofs of Propositions \ref{prop:3-9smooth}, \ref{prop:3-10-11smooth}, \ref{prop:4-8-smooth}, and \ref{prop:4-9-smooth}. We include supplementary computations useful, although not strictly necessary, for the proofs of Theorem \ref{thm:singular3-12} and Proposition \ref{prop:singularIsom}. The code can be found at the following github repository:

\begin{center}
    \url{https://github.com/dcorey2814/matroidRealizationSpaces}
\end{center}

\subsection*{Notation and conventions}
Given a set $E$, denote by $\binom{E}{d}$ the set of size $d$ subsets of $E$.  We write $[n] =\{1,\dots,n\}$. When $n<10$, we express subsets subsets of $[n]$ by juxtaposing their elements, e.g., we write $ijk$ in place of $\{i,j,k\}$,  $A\cup x$ in place of $A\cup \{x\}$, and $A\setminus x$ in place of $A\setminus \{x\}$. Finite tuples appear frequently when forming minors of a matrix. Let  $\lambda$ and $\mu$ be two tuples. We write $\lambda \setminus \mu$ for the tuple obtained by removing the elements of $\mu\cap \lambda$ from $\lambda$ while preserving the order and $\lambda\cup \mu$ for the tuple obtained by appending $\mu$ to the end of $\lambda$.   

\subsection*{Acknowledgements}
We thank Antony Della Vecchia, Chris Eur, Matteo Gallet, Michael Joswig, Lars Kastner, Lukas K\"{u}hne, Matt Larson, Benjamin Lorenz, and Marcel Wack for helpful discussions and feedback. We also thank the anonymous referee for their valuable feedback and suggestions.  DC is supported by the SFB-TRR project ``Symbolic Tools in Mathematics and their Application'' (project-ID 286237555), and DL is supported by "Facets of Complexity” (GRK 2434, project-ID 385256563).

\section{Background}

In this section we review preliminary material on Grassmannians, matroid strata, and realization spaces. In particular, for a given matroid, we demonstrate how to compute the coordinate rings of the associated strata in the Grassmannian, as well as its realization space.

\subsection{Grassmannian}\label{sec:grassmannian}
Let $\F$ be a field and let $d, n$ be positive integers satisfying $d \leq n$. The Grassmannian $\Gr(d,n;\F)$ parameterizes the $d$-dimensional subspaces of an $n$-dimensional $\F$-vector space.  We may represent $\Gr(d,n;\F)$ as a $d(n-d)$--dimensional projective variety via the Pl\"ucker embedding, defined as follows. Let $V\subset \F^n$ be a $d$-dimensional subspace and $A$ a full-rank $d\times n$ matrix whose row span is $V$. Given a sequence $\lambda = (i_1,\ldots,i_d)$ of elements in $[n]$, denote by $A_{\lambda}$ the determinant of the $d\times d$ matrix formed by the columns of $A$ indexed by $i_1,\ldots,i_d$ in that order. The $\lambda$-th \textit{Pl\"ucker coordinate} of $V$ is $p_{\lambda}(V) = A_{\lambda}$. If $\sigma$ is a permutation of the indices $(i_1,\ldots,i_d)$, then $p_{\sigma(\lambda)}(V) = \sgn(\sigma)p_{\lambda}(V)$ where $\sgn(\sigma)$ is the sign of the permutation $\sigma$; in particular $p_{\lambda}(V) = 0$ if $\lambda$ has repeated entries. If $\lambda$ is expressed as a subset of $[n]$, then $p_{\lambda}(V)$ is computed by taking the elements of $\lambda$ in increasing order. The \emph{Pl\"ucker embedding} is given by 
\begin{equation*}
    \Gr(d,n;\F)\hookrightarrow \P^{{n \choose d}-1}_{\F} \quad V \mapsto [p_{\lambda}(V):\lambda \in \textstyle{\binom{[n]}{d}}].
\end{equation*}
 Let  $S = \F[p_{\lambda} \, : \, \lambda \in \binom{[n]}{d}]$ be the homogeneous polynomial ring of $\P^{{n \choose d}-1}_{\F}$. The Grassmannian $\Gr(d,n;\F)$ is a closed subvariety of $\P^{{n\choose d}-1}_{\F}$. Its ideal in $S$ is called the \textit{Pl\"ucker ideal}, and is generated by 
 \begin{equation}
    \label{eq:pluecker-ideal-gens}
     \sum_{b=1}^{d+1} (-1)^{b} p_{i_1,\ldots,i_{d-1}, j_b} p_{j_1,\ldots, \widehat{j_{b}}, \ldots, j_{d+1} } 
 \end{equation}
where $(i_1,\ldots,i_{d-1})$ and $(j_{1},\ldots,j_{d+1})$ are sequences of elements of $[n]$; here, the variable $p_{\lambda}$, when $\lambda$ is a sequence rather than a set, is defined in a similar way as for $p_{\lambda}(V)$ above. See, e.g., \cite[Chapter~9.1]{FultonYT}.

\subsection{Matroids basics}\label{sec:matroids_basics}
We review basic concepts from matroid theory. For a full treatment, see \cite{Oxley}. Fix nonnegative integers $d\leq n$. A \textit{matroid} $\sQ$ is a pair $(E,\pB)$ where $E$ is a $n$-element set and $\pB$ is a nonempty subset of $\binom{E}{d}$ satisfying the \textit{basis exchange axiom}: For each pair of distinct elements $\lambda_1,\lambda_2$ of $\pB$ and $b_1\in \lambda_1\setminus \lambda_2$, there exists $b_2\in \lambda_2\setminus \lambda_1$ such that $(\lambda_1 \setminus b_1)\cup b_2$ lies in $\pB$. An element of $\pB$ is called a \textit{basis} of $\sQ$; for emphasis, we write $\pB(\sQ)$ for the set of bases of $\sQ$. A subset $\mu\subset E$ is \emph{independent} if it is contained in some basis of $\sQ$ (in particular, $\emptyset$ is independent), and is called \emph{dependent} otherwise. A \emph{circuit} of $\sQ$ is a dependent set $\gamma\subset E$ such that every proper subset of $\gamma$ is independent in $\sQ$. 

The integer $d$ above is called the \textit{rank} of $\sQ$ and the set $E$ is called the \textit{ground set} of $\sQ$. A \textit{$(d,E)$--matroid} is a rank $d$ matroid with ground set $E$; when $E=[n]:=\{1,\ldots,n\}$, we simply call this a  \textit{$(d,n)$--matroid}. We obtain a \emph{rank function} $\rho_{\sQ}: E\to \N$, whose values are given by
\begin{equation*}
    \rho_{\sQ}(\lambda) =\text{max}\{|\lambda\cap B| : B\in\pB(\sQ)\},
\end{equation*}
\noindent for any $\lambda\subset E$.  A rank $k$ subset $\eta\subset E$ is a \emph{flat} of $\sQ$ if $\rho_{\sQ}(\eta\cup a)=k+1$ for any $a\in E\setminus \eta$. If $\sQ$ has rank $d$, the rank $d-1$ flats are known as \emph{hyperplanes}. 
While the definition in terms of bases is most relevant to our work, matroids admit many cryptomorphic characterizations. In particular, one can define a matroid axiomatically via independent sets, circuits, rank functions, flats, or hyperplanes.

Let $\sQ$ be a $(d,E)$--matroid. A \textit{loop} is an element of $E$ not contained in any basis of $\sQ$, and a \textit{coloop} is an element of $E$ contained in every basis of $\sQ$.  Two distinct elements $a,b\in E$ are \textit{parallel} if $a$ and $b$ are not loops and $\{a,b\}\not\subset \lambda$ for all $\lambda \in \pB(\sQ)$. A matroid is \textit{simple} if it has no loops or parallel elements. Given two matroids $\sQ_1$ and $\sQ_2$ on the ground sets $E_1$ and $E_2$ respectively, the \textit{direct sum} of $\sQ_1$ and $\sQ_{2}$, denoted $\sQ_{1}\oplus \sQ_{2}$, is the matroid on $E_1\sqcup E_2$ with bases
\begin{equation*}
    \pB(\sQ_{1} \oplus \sQ_{2}) = \{ \lambda_1\cup \lambda_2 \, : \, \lambda_1\in \pB(\sQ_1), \lambda_2\in \pB(\sQ_2) \}.
\end{equation*}

\noindent A matroid is \textit{connected} if it does not admit a decomposition into a nontrivial direct sum. 

Standard operations in matroid theory are those of duality, deletion, and contraction.  The \textit{dual} of $\sQ$ is the $(|E|-d,E)$ matroid $\sQ^{\vee}$ whose bases are
\begin{equation*}
    \pB(\sQ^{\vee}) = \{E\setminus \lambda \, : \, \lambda \in \pB(\sQ)\}.
\end{equation*}
Clearly $(\sQ^{\vee})^{\vee}=\sQ$. Let $\eta\subset E$.  The \emph{deletion} by $\eta$ is the matroid $\sQ\setminus\eta$ with ground set $E\setminus \eta$ whose independent sets are the independent sets of $\sQ$ contained in $E\setminus \eta$. The \textit{restriction} of $\sQ$ to $\eta$ is the matroid  $\sQ_{|\eta} = \sQ \setminus (E\setminus \eta)$; its ground set is $\eta$.  The \emph{contraction} by $\eta$ is $\sQ/\eta = (\sQ^{\vee}\setminus\eta)^{\vee}$; its ground set is $E\setminus \eta$. In \S\ref{sec:matroidStrataAndBasicOperations} we will discuss polynomial maps between affine schemes induced by the above operations, and the analogous "reverse" operations for deletion and contraction.

\subsection{Matroid strata and their coordinate rings}\label{sec:coordrings_tsc}

Let $V\subset \F^{n}$ be a $d$-dimensional linear subspace. Its matroid $\sQ(V)$ is the $(d,n)$--matroid with bases
\begin{equation*}
    \pB(\sQ(V)) = \{\lambda \in \textstyle{\binom{[n]}{d}} \, : \, p_{\lambda}(V) \neq 0\}.
\end{equation*}
A $(d,E)$--matroid is \emph{$\F$-realizable} if there exists a $d$-dimensional linear subspace $V\subset \F^{n}$ such that $\sQ \cong \sQ(V)$. Such a $V$ is called a \textit{realization} of $\sQ$. 

Let $\sQ$ be a  $\F$-realizable $(d,n)$--matroid $\sQ$. Its \textit{matroid stratum}, as a set, is
\begin{equation*}
\Gr(\sQ; \F) = \{V\in\Gr(d,n;\F) \, | \,  p_{\lambda}(V) \neq 0 \text{ if and only if } \lambda\in \pB(\sQ)\}.   
\end{equation*}
The Grassmannian $\Gr(d,n;\F)$ decomposes into matroid strata (although this is not a \textit{Whitney stratification}, see \cite[\S 5]{GelfandGoreskyMacPhersonSerganova}). 

We describe $\Gr(\sQ;\F)$ as an affine scheme. More precisely, we review the presentation for the coordinate ring of $\Gr(\sQ;\F)$ in terms of matrix coordinates.  Fix a \textit{reference basis} $\lambda_{0} \in \pB(\sQ)$.  To streamline the exposition, we assume that $\lambda_{0} = \{1,\ldots,d\}$. Define the polynomial ring 
\begin{equation*}
    B =  \F[x_{ij} \, : \, 1\leq i \leq d,\, 1\leq j \leq n-d].
\end{equation*}
and the matrix of variables
\begin{equation*}
A = \begin{bmatrix}
    1 &0 &\dots &0 &x_{11} &\dots &x_{1,n-d}\\
    0 &1 &\dots &0 &x_{21} &\dots &x_{2,n-d}\\
    \vdots &\vdots &\dots &\vdots &\vdots &\dots &\vdots\\
    0 &0 &\dots &1 &x_{d1} &\dots &x_{d,n-d}\\
\end{bmatrix}.
\end{equation*}
Recall from \S\ref{sec:grassmannian} that, given a $d\times n$ matrix $A$ and a tuple $\lambda = (i_1,\ldots,i_d)$ of  elements of $[n]$, we write $A_{\lambda}$ for the determinant of the $d\times d$ submatrix formed by the columns of $A$ indexed by $\lambda$ in order. If $\lambda$ is given as a set, $A_{\lambda}$ is computed by taking the columns in increasing order. If $|\lambda| \neq d$, or if $A_{\lambda}$ has a repeated entry, then $A_{\lambda} = 0$. The variable $x_{ij}$ can be expressed as a maximal minor:
\begin{equation*}
    x_{ij} = (-1)^{d+i} A_{[d]\setminus i \cup (d+j)}.
\end{equation*}
Define the polynomial ring
\begin{equation*}
B_{\sQ} = \F[x_{ij}\, : \, ([d]\setminus i)\cup (d+j) \in \pB(\sQ)]
\end{equation*}
and the surjective ring homomorphism
\begin{equation}
\label{eq:piQ}
    \pi_{\sQ}:B \to B_{\sQ}, \hspace{10pt} \pi_{\sQ}(x_{ij}) = 
    \begin{cases}
    x_{ij} & \text{ if } ([d]\setminus i)\cup (d+j) \in \pB(\sQ), \\
    0 & \text{ otherwise.} 
    \end{cases}   
\end{equation}
Define the ideal $I_{\sQ}$, multiplicative semigroup $U_{\sQ}$, and ring $S_{\sQ}$
\begin{equation*}
    I_{\sQ} = \langle \pi_{\sQ}(A_\lambda): \lambda \notin\pB(\sQ)\rangle, 
    \quad
    U_{\sQ} = \langle \pi_{\sQ}(A_{\lambda}):\lambda\in \pB(\sQ)\rangle_{\smgp}, 
    \quad 
    S_{\sQ} = U_{\sQ}^{-1} B_{\sQ} / I_{\sQ}. 
\end{equation*}
Thus, $\Gr(\sQ;\F)$ is the affine scheme $\Spec(S_{\sQ})$. When $\F$ is algebraically closed, the closed points of $\Gr(\sQ;\F)$ correspond to the realizations of $\sQ$ as defined previously.

\subsection{Matroid realization spaces and their coordinate rings}
\label{sec:coordrings_realization_spaces}
A \textit{hyperplane arrangement} is a sequence of $n$ hyperplanes in $\P^{d-1}$. Given a hyperplane arrangement  $(h_1,\ldots,h_n)$ such that $\bigcap^{n}_{i=1}h_i=\emptyset$, we obtain a matroid $\sQ$ whose ground set is $[n]$ and with bases given by
\begin{equation*}
\pB(\sQ) = \{ \lambda \in \textstyle{\binom{[n]}{d}} \, : \, \bigcap_{i\in \lambda} h_i = \emptyset \}.
\end{equation*}
A loopless matroid is the matroid of a hyperplane arrangement in $\P^{d-1}_{\F}$ if and only if it is $\F$-realizable. 

Let $\sQ$ be a $\F$-realizable and connected $(d,n)$--matroid. (The study of realization spaces of disconnected matroids readily reduces to the connected case. We refer the reader to \cite[\S~9]{KatzMatroid} which summarizes the discussion in \cite[\S~1.6]{Lafforgue2003}).  The \textit{realization space}  of $\sQ$ is a scheme $\pR(\sQ;\F)$ that parameterizes equivalence classes of hyperplane arrangements; here, two hyperplane arrangements $(g_1,\ldots,g_n)$ and $(h_1,\ldots,h_n)$ are equivalent if there is a $\phi \in \PGL_{d}(\F)$ such that $\phi(g_i) = h_i$ for $i=1,\ldots,n$. 

Denote by $T\subset \PGL_{n}(\F)$ the diagonal subtorus, which is isomorphic to $(\F^*)^{n}/\F^*$. The torus $T$ acts on linear subspaces of $\F^{n}$ by scaling coordinates. This induces an action of $T$ on $\Gr(d,n;\F)$, and on each matroid stratum $\Gr(\sQ;\F)$. By \cite[Th\'eor\`eme~1.11]{Lafforgue2003}, if $\sQ$ is connected in the sense of \S\ref{sec:matroids_basics}, then the action $T \curvearrowright \Gr(\sQ;\F)$ is free and we have an isomorphism
\begin{equation}
\label{eq:gelfand-macpherson}
    \pR(\sQ;\F) \cong \Gr(\sQ;\F)/T.
\end{equation}
This is a matroidal refinement of the \textit{Gelfand--MacPherson correspondence} \cite{GelfandMacPherson}.  
In particular, $\Gr(\sQ;\F)$ and $\pR(\sQ;\F)$ have the same singularity types and connected components. 
The scheme $ \pR(\sQ;\F)$ is affine, and the general procedure to obtain its coordinate ring is similar to that of a matroid stratum with a few subtle differences that we describe now.

We now describe the coordinate ring of $\pR(\sQ;\F)$.  So as to not overburden the notation with extra decorations, we use symbols similar to those for the coordinate ring of a matroid stratum. In later sections, we hope that the context clarifies the usage. Assume that $\sQ$ has a a circuit $\gamma_{0}$ of size $d+1$, which we call our \textit{reference circuit}.\footnote{Such a circuit need not exist, see \S \ref{sec:4n} for some examples. Nevertheless, this hypothesis greatly simplifies the following construction. }
After applying a suitable permutation, we may assume that $\gamma_0 = \{1,2,\ldots,d+1\}$. Since $\gamma_0$ is a circuit, the subset $\lambda_0 = \{1,2,\ldots,d\}$ is a basis of $\sQ$. Define 
\begin{equation*}
    B = \F[x_{ij} \, : \, 1 \leq i \leq d, \, 1\leq j \leq n-d-1]
\end{equation*}
and the matrix of variables
\begin{equation*}
A = \begin{bmatrix}
    1 &0 &\dots &0 & 1 &x_{11} &\dots &x_{1,n-d-1}\\
    0 &1 &\dots &0 & 1 &x_{21} &\dots &x_{2,n-d-1}\\
    \vdots &\vdots & \dots &\vdots &\vdots &\vdots &\dots &\vdots\\
    0 &0 &\dots &1 &1 &x_{d1} &\dots &x_{d,n-d-1}\\
\end{bmatrix}.
\end{equation*}
Define the function $\mu:[n-d-1] \to [d]$, which depends on $\sQ$ and $\lambda_0$, by
\begin{equation*}
    \mu(j) = \max(i\in [d] \, : \, ([d] \setminus i)\cup (d+j+1) \in \pB(\sQ)).
\end{equation*}
Define the polynomial ring
\begin{equation*}
B_{\sQ} = \F[x_{ij}\, : \, ([d] \setminus i)\cup (d+j+1) \in \pB(\sQ) \text{ and } i\neq \mu(j)]
\end{equation*}
and the surjective ring homomorphism $\pi_{\sQ}:B \to B_{\sQ}$ by
\begin{equation*}
    \pi_{\sQ}(x_{ij}) = 
    \begin{cases}
    x_{ij} & \text{ if } ([d]\setminus i) \cup (d+j+1)\in \pB(\sQ) \text{ and } i\neq \mu(j), \\
    1 & \text{ if } i = \mu(j), \\
    0 & \text{ otherwise.} 
    \end{cases}   
\end{equation*}
Define the ideal $I_{\sQ}$ and multiplicative semigroup $S_{\sQ}$ by
\begin{equation*}
    I_{\sQ} = \langle \pi_{\sQ}( A_\lambda): \lambda \notin \pB(\sQ)\rangle, \hspace{20pt} 
    U_{\sQ} = \langle \pi_{\sQ}(A_{\lambda}):\lambda\in \pB(\sQ)\rangle_{\smgp}.
\end{equation*}
The coordinate ring of the affine scheme $\pR(\sQ;\F)$ is
\begin{equation*}
\F\left[\pR(\sQ;\F)\right]\cong U_{\sQ}^{-1}B_{\sQ}/I_{\sQ}.
\end{equation*}

\begin{example}
\label{ex:3-9}
We demonstrate the procedure in \ref{sec:coordrings_realization_spaces} to compute the coordinate ring of the $(3,9)$-matroid $\sQ$, where the nonbases of $\sQ$ are given by the colinearities in Figure \ref{fig:3_9_example}. We then show that $\pR(\sQ,\C)$ is smooth.

\begin{figure}
    \centering
    \includegraphics[height = 4cm]{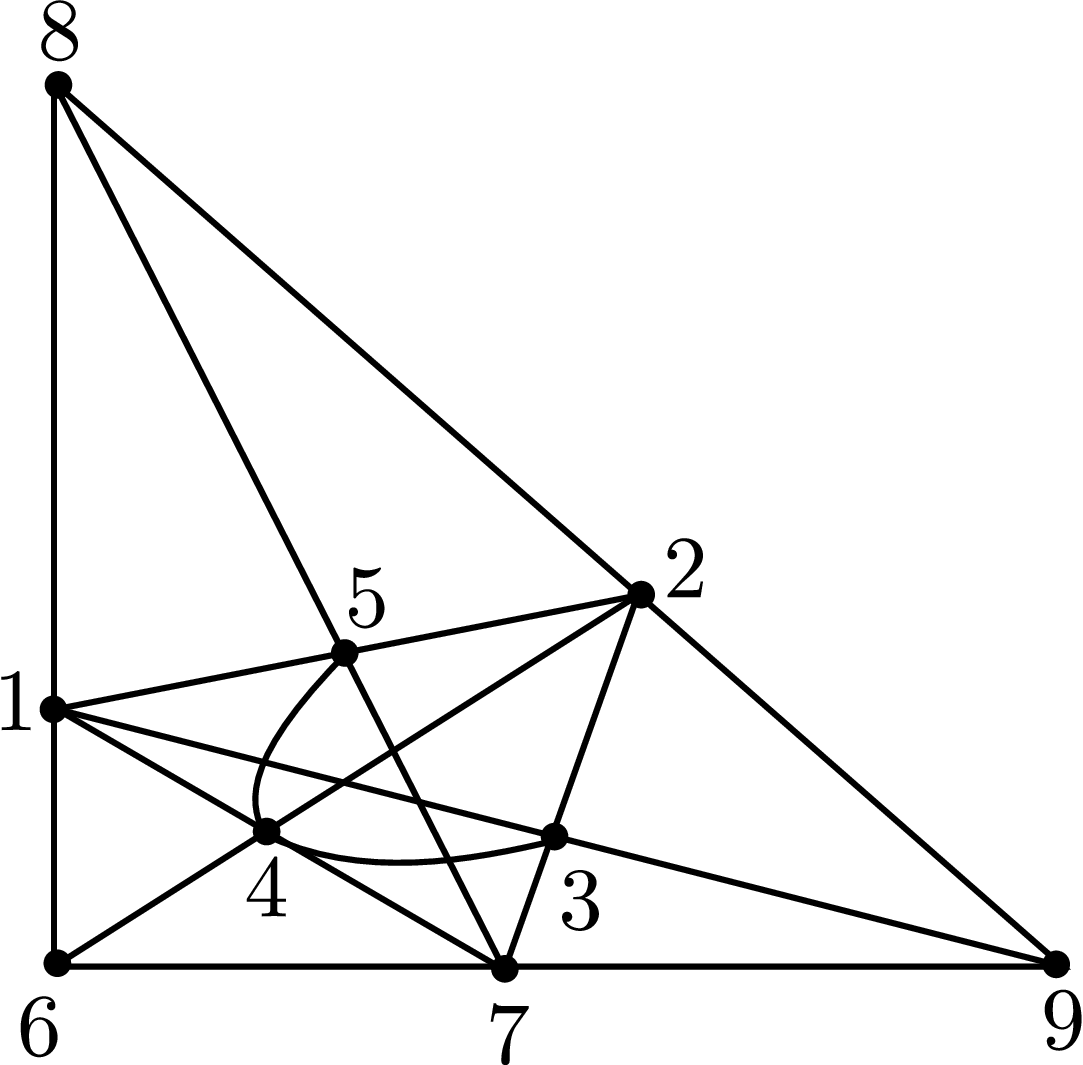}
    \caption{A projective realization of the $(3,9)$--matroid from Example \ref{ex:3-9}}
    \label{fig:3_9_example}
\end{figure}
The matrix of indeterminants below realizes $\sQ$, provided that the minors indexed by nonbases are zero, and those indexed by bases are nonzero. 

\begin{equation*}
A = \begin{bmatrix}
    1   &0   &0   &1   &x_1   &x_2    &0   &x_5   &x_7\\
    0   &1   &0   &1    &1   &x_3   &x_4   &x_6    &0\\
    0   &0   &1   &1   & 0    &1    &1    &1    &1
\end{bmatrix}.
\end{equation*}

Hence, we have $\pR(\sQ,\C)\cong \Spec\left(U^{-1}_{\sQ}B_{\sQ}/I_{\sQ}\right)$ where $B_{\sQ} = \C[x_1,\ldots,x_7]$,  $I_{\sQ}$ is the ideal
   \begin{equation*}
   I_{\sQ} = \left\langle \begin{array}{c}
       x_1 - 1, \, 
       x_2 - 1, \,
       x_3 - x_7 - 1, \, 
       x_4 - 1, \,
       x_5 - x_7, \,
       x_6 - x_7 - 1, \, 
       x^2_{7} + 1
   \end{array}\right\rangle
   \end{equation*}
   and $U_{\sQ}$ is the multiplicative semigroup generated by 
   {\footnotesize
   \begin{gather*}
x_1,
x_2,
x_3,
x_4,
x_5,
x_6,
x_7,
x_3 - 1,
x_5 - 1,
x_6 - 1,
x_7 - 1, 
x_2 - x_3,
x_2 - x_5,
x_2 - x_7, 
x_3 - x_4, 
x_4 - x_6, 
x_5 - x_6, \\
x_1 + x_7 - 1,
x_1x_3 - x_1 - x_2 + 1,
x_1x_3 - x_1x_6 - x_2 + x_5,
x_1x_3 - x_2, 
x_1x_3 - x_2 + x_7,
x_1x_3 - x_1x_4 - x_2, \\
x_1x_4 - x_1 + 1,
x_1x_4 + x_7,
x_1x_6 - x_1 - x_5 + 1,
x_1x_6 - x_5, 
x_1x_6 - x_5 + x_7,
x_2 + x_3x_7 - x_3 - x_7, \\
x_2x_4 - x_2 + x_3 - x_4,
x_2x_4 - x_2x_6 + x_3x_5 - x_4x_5, 
x_2x_6 - x_2 - x_3x_5 + x_3 + x_5 - x_6,
x_2x_6 - x_3x_5, \\
x_2x_6 - x_3x_5 + x_3x_7 - x_6x_7, 
x_4x_5 - x_4 - x_5 + x_6, 
x_4x_5 - x_4x_7 + x_6x_7,
x_4x_7 - x_4 - x_7,
x_5 + x_6x_7 - x_6 - x_7.
   \end{gather*}
   }
Next we apply the ring homomorphism $\phi:B_{\sQ}\to \C[x]$ given by
\begin{equation*}
\begin{array}{ccccccc}
     x_1\mapsto 1 \, &x_2\mapsto 1  \, &x_3\mapsto x+1 &x_4\mapsto 1 \, &x_5\mapsto x \, &x_6\mapsto x+1  &x_7\mapsto x.
\end{array}
\end{equation*}
With relations from $I_{\sQ}$ and $U_{\sQ}$, the map $\phi$ induces an isomorphism 
\begin{equation*}
U^{-1}_{\sQ}B_{\sQ}/I_{\sQ} \cong U^{-1}\C[x]/\langle x^2+1\rangle \text{ where } U = \langle x+1,x-1,x \rangle_{\smgp}.
\end{equation*} 
Since $x+1,x-1,x$ are units in $\C[x]/\langle x^2+1\rangle$, we have $\pR(\sQ,\C)\cong \Spec(\C[x]/\langle x^2+1\rangle)$. The coordinate ring of $\pR(\sQ;\C)$ is isomorphic to $\C \times \C$, and so $\pR(\sQ;\C)$ is a reduced 0-dimensional scheme consisting of 2 closed points. Thus $\pR(\sQ;\C)$ is smooth. 

The homomorphism $\phi$ and induced isomorphism mirror \cite[Algorithm~6.12]{CoreyLuber}, a technique we use in later examples and throughout the computational aspects of our proofs. 
\end{example}

\section{Matroid strata and basic operations}
\label{sec:matroidStrataAndBasicOperations}

We are primarily interested in determining which matroids have smooth realization spaces. By the isomorphism in Formula \eqref{eq:gelfand-macpherson}, we can deduce this from properties of matroid strata. In this section, we describe how matroid strata behave under the operations introduced in \S\ref{sec:matroids_basics}. Furthermore, we review the notions of principal extension and coextension. Realization spaces for matroids constructed via these operations can be studied by applying knowledge about matroids of smaller rank, or on a smaller ground set. The first two propositions are well-established, see \cite[Proposition~9.4]{KatzMatroid}.

First consider the direct sum and dual operations. If $V_1$ and $V_2$ are realizations of $\sQ_1$ and $\sQ_2$ respectively, then $V_1\oplus V_2$ is a realization of the direct sum $\sQ_1\oplus \sQ_2$. More generally, we have the following. 
\begin{proposition}
\label{prop:directsum}
    If $\sQ_1$ and $\sQ_2$ are $\F$-realizable, then $\sQ_1 \oplus \sQ_2$ is $\F$-realizable and
    \begin{equation*}
        \Gr(\sQ_1 \oplus \sQ_2; \F) \cong \Gr(\sQ_1;\F) \times \Gr(\sQ_2;\F).
    \end{equation*}
\end{proposition}

\noindent If $V\subset \F^{n}$ is a realization of a matroid $\sQ$, then $V^{\perp} = \{\sv\in (\F^{n})^{\vee} \, : \, \bk{\su}{\sv} = 0 \text{ for all } \su\in V \}$ is a realization of its dual $\sQ^{\vee}$.  More generally, we have the following. 
\begin{proposition}
\label{prop:duality}
    If $\sQ$ is $\F$-realizable, then so is $\sQ^{\vee}$ and
    \begin{equation*}
        \Gr(\sQ^{\vee};\F) \cong \Gr(\sQ;\F).
    \end{equation*}
\end{proposition}

Next we consider deletion and contraction. Let $\sQ$ be a $\F$-realizable matroid on $E$ and let $a\in E$. Set $E' = E\setminus a$ and $\sQ' = \sQ\setminus a$. We define the morphism $\varphi_{\sQ,\sQ'}:\Gr(\sQ;\F) \to \Gr(\sQ';\F)$ by the  composition
\begin{equation*}
 \Gr(\sQ;\F) \to \Gr(\sQ';\F) \times \Gr(\sQ/E';\F) \to  \Gr(\sQ';\F).
\end{equation*}
Here, the second map is the projection onto the first factor and the first map is the face morphism in  \cite[Proposition I.6]{Lafforgue2003} and \cite[Proposition~3.1]{CoreyGrassmannians}, which we recall in \S \ref{sec:inverse_limits_tight_spans} (more precisely, this corresponds to a face contained in the boundary of the hypersimplex defined by $x_a=0$). The morphism $\varphi_{\sQ,\sQ'}$ need not be smooth, see \cite[Example~8.4]{CoreyGrassmannians} and Example \ref{ex:morphismNotSmooth}. A key goal is to isolate conditions where this morphism is smooth and surjective with connected fibers. Two broad classes of examples come from principal extensions and coextensions of matroids, as we now describe. 

Let $\sQ'$ be a $(d,E')$--matroid, and let $\eta'$ be a flat of $\sQ'$. The \textit{principal extension} of $\sQ'$ by $a$ into $\eta'$ is the $(d,E)$--matroid $\sQ = \pe{\sQ'}{\eta'}{a}$ where $E = E'\sqcup a$ and 
\begin{equation*}
    \pB(\sQ) = \pB(\sQ') \cup \{(\lambda' \setminus b )\cup a \, : \, \lambda'\in \pB(\sQ') \text{ and } b\in \lambda'\cap \eta'\}.
\end{equation*}

\noindent The principal extension above is called a \textit{free extension} if $\eta' = E'$.  If $\F$ is an infinite field and $\sQ'$ is $\F$-realizable, then $\sQ$ is also $\F$-realizable \cite[Lemma~2.1]{MayhewNewmanWhittle}. In the next proposition, we prove that principal extension preserves the singularity type. Here is the intuitive idea. All projective realizations of $\sQ$  are obtained by freely placing, in a projective realization of $\sQ'$, a point on the linear subspace spanned by the flat $\eta'$ outside of finitely many hypersurfaces, and hence $\Gr(\sQ;\F)$ is an open subscheme of $\Gr(\sQ';\F) \times (\G_{m})^k$ where $k$ is the rank of $\eta$. 

\begin{proposition}
\label{prop:principalExtension}
Let $\sQ, \sQ'$ be $\F$-realizable matroids such that $\sQ$ is a principal extension of $\sQ'$. Then the morphism 
\begin{equation*}
    \varphi_{\sQ,\sQ'}:\Gr(\sQ;\F) \to \Gr(\sQ';\F)
\end{equation*}
is smooth. If $\F=\C$ then, over closed points, this morphism is surjective with connected fibers. 
\end{proposition}

\begin{lemma}
\label{lem:expand}
Define the $d\times n$ matrix of variables
    \begin{equation}
    \label{eq:A}
A = \begin{bmatrix}
    1 &0 &\cdots &0 &x_{11} &\cdots &x_{1,n-d-1} & y_1\\
    0 &1 &\cdots &0 &x_{21} &\cdots &x_{2,n-d-1} & y_2 \\
    \cdots &\cdots &\cdots &\cdots &\cdots &\cdots &\cdots & \cdots\\
    0 &0 &\cdots &1 &x_{d1} &\cdots &x_{d,n-d-1} & y_d
    \end{bmatrix}.
    \end{equation}
Given a sequence  $\lambda' = (i_1,\ldots,i_{d-1})$ of integers in $[n-1]$, we have 
\begin{equation*}
    A_{\lambda'\cup n} =  \sum_{b=1}^{d} A_{\lambda' \cup b}\, y_{b}.
\end{equation*}
\end{lemma}
\noindent This lemma follows from the  quadratic Pl\"ucker relation in \eqref{eq:pluecker-ideal-gens} where  $(j_1,\ldots, j_{d}, j_{d+1}) = (1,\ldots,d,n)$, and the formula
\begin{equation*}
y_{b} = A_{1,\ldots,b-1,n,b+1,\ldots,d} = (-1)^{d-b}A_{1,\ldots,b-1,b+1,\ldots,d,n} 
\quad \text{ for } \quad 1\leq b\leq d.    
\end{equation*}

\begin{proof}[Proof of Proposition \ref{prop:principalExtension}]
Let $\eta'$ be the flat of $\sQ^{\prime}$ such that $\sQ = \sQ^{\prime}+_{\eta'} a$. 
We show that we have a commutative diagram
\begin{equation*}
    \begin{tikzcd}
 \Gr(\sQ;\F)\arrow[dr, "\varphi_{\sQ,\sQ'}"'] \arrow[r, hook, "\iota"] &\Gr(\sQ';\F) \times (\G_{m})^k \arrow[d, "\pr_1"] 
\\
& \Gr(\sQ';\F) 
\end{tikzcd}
\end{equation*}
such that $\iota$ is an open affine immersion, $\pr_1$ is the projection onto the first factor, and $k$ is the rank of $\eta$.  After applying a suitable isomorphism, we may assume that $E=[n]$ is the ground set of $\sQ$, $a=n$, $[d]$ is a basis of $\sQ'$, and $[k] = \eta' \cap [d]$.  Let $A$ be the matrix from Formula \eqref{eq:A}, where we take $[d]$ to be our reference basis. Throughout, let $A_{\lambda}(\sQ) = \pi_{\sQ}(A_{\lambda})$ where $\pi_{\sQ}$ is the ring map from Formula \eqref{eq:piQ}.

\medskip

\noindent \textbf{Claim:} If $\lambda' \in  \binom{[n-1]}{d}$ such that $\lambda' \cup n \notin \pB(\sQ)$, then $A_{\lambda' \cup n}(\sQ) \in I_{\sQ'} \cdot B_{\sQ}$. 

\medskip

\noindent Suppose $\lambda'$ and $n$ are as in the claim. Since $([d]\setminus \ell) \cup n \notin \pB(\sQ)$ for $k<\ell\leq d$, we have $\pi_{\sQ}(y_\ell) = 0$, and hence 
\begin{equation*}
    A_{\lambda' \cup n}(\sQ) =  \sum_{b\in [k]\setminus \lambda'}  A_{\lambda'\cup b}(\sQ') \, y_{b}
\end{equation*}
by Lemma \ref{lem:expand}. Because $\lambda' \cup n$ is not a basis of $\sQ$, the set $\lambda'\cup b$ (for $b \in [k]\setminus \lambda'$) is not a basis of $\sQ'$. So $A_{\lambda'\cup n}(\sQ) \in I_{\sQ'} \cdot B_{\sQ}$, as required. 

As a consequence of this claim, the ideal $I_{\sQ}$ is the extension of $I_{\sQ'} \subset B_{\sQ'}$ to $B_{\sQ}$. So  $ U_{\sQ}^{-1}(S_{\sQ'}\otimes \F[y_1^{\pm}, \ldots, y_{k}^{\pm}]) \to S_{\sQ}$ is an isomorphism, and hence we have an open immersion $\iota:\Gr(\sQ;\F) \hookrightarrow \Gr(\sQ';\F)\times (\G_m)^{k}$. This implies that $\varphi_{\sQ,\sQ'}$ is smooth and its nonempty fibers are connected. 

Finally, for $\F=\C$, we must show that $\varphi_{\sQ,\sQ'}$ is surjective on closed points.   Let $V'\subset \C^{n-1}$ be a linear subspace with $\sQ(V') = \sQ'$ and $A'$ be a $\C$-valued, $d\times (n-1)$--matrix whose row span is $V'$. Choose $y_1,\ldots,y_k$ algebraically independent over $\Q$. Because the polynomials $f\in S_{\sQ}$ have $\Z$-coefficients, we have that $f(x_{ij},y_{\ell})\neq 0$ for all $f\in U_{\sQ}$. Thus, with $y = [y_1,\ldots,y_k,0,\ldots,0]^T \in \C^{d}$, the row span $V$ of $[A'|y]$ is a $\C$-realization of $\sQ$ and $\varphi_{\sQ,\sQ'}(V) = V'$. 
\end{proof}

Just as principal extension reverses matroid deletion, principal coextension reverses matroid contraction. Let $\sQ^{\prime}$ be a matroid with ground set $E'$ and $\eta'$ a flat of $(\sQ^{\prime})^{\vee}$. The \textit{principal coextension} of $\sQ'$ by $a$ into $\eta'$ is the matroid $\sQ$ with ground set $E=E^{\prime} \sqcup a$ defined by 
\begin{equation*}
    \sQ = (\pe{(\sQ')^{\vee}}{\eta'}{a})^{\vee}.
\end{equation*}
When $\eta'=E'$, this is called a \textit{free coextension}, and the rank of $\sQ$ is one more than that of $\sQ'$.  Define the morphism $\varphi_{\sQ,\sQ'}:\Gr(\sQ;\F) \to \Gr(\sQ';\F)$ by the composition
\begin{equation*}
    \Gr(\sQ;\F) \to \Gr(\pe{(\sQ')^{\vee}}{\eta'}{a};\F) \to \Gr((\sQ')^{\vee};\F) \to \Gr(\sQ';\F) 
\end{equation*}
where the first and third morphisms are the duality isomorphism in Proposition \ref{prop:duality}, and the middle one is the morphism in Proposition \ref{prop:principalExtension}. The following is an immediate consequence of these two propositions.

\begin{proposition}\label{prop:coextension}
    Let $\sQ,\sQ'$ be $\F$-realizable matroids such that $\sQ$ is a principal coextension of $\sQ'$. Then the morphism 
    \begin{equation*}
        \Gr(\sQ;\F) \to \Gr(\sQ';\F)
    \end{equation*}
    is smooth. If $\F = \C$, then this morphism is also surjective with connected fibers. 
\end{proposition}

Recall that a hyperplane of a rank $d$ matroid is a flat of rank $d-1$. 
We denote the set of all hyperplanes of $\sQ$ by $\pH(\sQ)$. 
Given $a\in E$, set 
\begin{equation*}
        \pZ^1(\sQ,a) = \{ H \in \pH(\sQ) \, : \, a\in H \text{ and }\, a \text{ is not a coloop of } \sQ_{|H} \}.
\end{equation*}

\begin{proposition}
\label{prop:2planes}
    Let $\sQ$ be a $\F$-realizable,  connected, rank $d\geq 3$ matroid on $E$, and let $a\in E$.   If $\pZ^1(\sQ,a)$
    has at most 2 elements, then the morphism $\Gr(\sQ;\F) \to \Gr(\sQ\setminus a;\F)$ is smooth. If $\F = \C$ then, over closed points, this morphism is surjective with connected fibers. 
\end{proposition}

\noindent 

 When $|\pZ^1(\sQ,a)|\leq 1$, Proposition \ref{prop:principalExtension} applies, see Lemma \ref{lem:01planes} below. If $\pZ^1(\sQ,a)=\{H_1,H_2\}$ but $|H_1\cap H_2|$ has too few elements, then Proposition \ref{prop:principalExtension} does not apply. Nevertheless, the idea is similar. The projective realizations of $\sQ$ are obtained by freely placing, in a projective realization of $\sQ'$, a point on the codimension-2 subspace lying at the intersection of the hyperplanes spanned by $H_1$ and $H_2$ (outside of finitely many hypersurfaces). So $\Gr(\sQ;\F)$ is isomorphic to an open subscheme of $\Gr(\sQ;\F)\times (\G_{m})^{d-2}$. The bound $|\pZ^1(\sQ,a)|\leq 2$ is essential, see Example \ref{ex:morphismNotSmooth}. 

Before proving this proposition, we introduce the following lemma that will assist in detecting when $\sQ$ is a principal extension of $\sQ\setminus a$ by $a$. Its proof will  make use of the closure operator $\cl_{\sQ}$ of the matroid $\sQ$. Given a subset $A\subset E$, its \textit{closure} $\cl_{\sQ}(A)$ is the smallest flat containing $A$.  See \cite[Chapter 1]{Oxley} for details.

\begin{lemma}
\label{lem:01planes}
    Let $\sQ$ be a rank-$d$ matroid, let $\eta$ be a nonempty flat of $\sQ$, and let $a\in \eta$ such that $a$ is not a coloop of $\sQ_{|\eta}$. Set $\eta' = \eta\setminus a$ and $\sQ' = \sQ\setminus a$. 
    \begin{enumerate}
        \item We have $\pB(\sQ) \subset \pB(\pe{\sQ'}{\eta'}{a})$.
        \item If $\pZ^1(\sQ_{|\eta},a) = \emptyset$ then $\sQ = \pe{\sQ'}{\eta'}{a}$.
        \item If $\pZ^1(\sQ,a)$ consists of $r$ distinct hyperplanes $H_1,\ldots, H_{r}$ (for some $r\geq 1$) and, with $\eta = H_1\cap \cdots \cap H_r$, we have
        \begin{equation*}
            \rho_{\sQ}(H_1\cap \cdots \cap H_r\setminus a) = \rho_{\sQ}(H_1\cap \cdots \cap H_r) = d-r,
        \end{equation*}
        then $\sQ = \pe{\sQ'}{\eta'}{a}$; here $\rho_{\sQ}$ is the rank function of $\sQ$.
    \end{enumerate}
\end{lemma}

\begin{proof}
We prove statement (1) in the following form: if $\pB(\sQ) \not\subset \pB(\pe{\sQ'}{\eta'}{a})$  then $a$ is a coloop of $\sQ_{|\eta}$. To that end, suppose $\lambda\in \pB(\sQ)$ contains $a$ such that $(\lambda \setminus a) \cup b$ is not a basis of $\sQ'$ for all $b\in \eta'$. 
The set  $H = \cl_{\sQ}(\lambda \setminus a)$ is a hyperplane of $\sQ$ that contains $\eta'$ and does not contain $a$. Since  $H\cap \eta$ is a flat of $\sQ_{|\eta}$ and $\eta'=H\cap \eta$,  we have that $a$ is a coloop of $\sQ_{|\eta}$.

Next consider statement (2). Suppose $\sQ\neq \pe{\sQ'}{\eta'}{a}$. By part (1), there is a  $\lambda'\in \pB(\sQ')$ and $b\in \lambda'\cap \eta'$ such that $\lambda = (\lambda'\setminus b) \cup a \in \pB(\pe{\sQ'}{\eta'}{a})\setminus \pB(\sQ)$. So $H = \cl_{\sQ}(\lambda)$ is a hyperplane of $\sQ$ that contains $a$ and does not contain $b$. The intersection $H\cap \eta$ is a flat of $\sQ_{|\eta}$ and $H\cap \eta \neq \eta$. Therefore, there is a hyperplane $F$ of $\sQ_{|\eta}$ containing $H\cap \eta$. As $a$ is contained in $F$ and is not a coloop of $\sQ_{|F}$, we have that $F\in \pZ^{1}(\sQ_{\eta},a)$, as required. 

Finally consider statement (3). Clearly $a$ is not a coloop of $\sQ_{|\eta}$. Let $F$ be a hyperplane of $\sQ_{|\eta}$, i.e., $F$ is a flat of $\sQ$ contained in $\eta$ and has rank $d-r-1$, and suppose $a\in F$. By \cite[Proposition 1.7.8]{Oxley} there is a hyperplane $H_{r+1}$ of $\sQ$ different from $H_1,\ldots,H_r$ such that $F \subset H_{r+1}$. By the hypothesis on $\pZ^{1}(\sQ,a)$, we have that $a$ is a coloop of $\sQ_{|H_{r+1}}$, and so $a$ is a coloop of $\sQ_{|F}$. Therefore $\pZ(\sQ_{|\eta},a) = \emptyset$, and so $\sQ = \pe{\sQ'}{\eta'}{a}$ by part (2). 
\end{proof}

\begin{proof}[Proof of Proposition \ref{prop:2planes}]
By Lemma \ref{lem:01planes}, if $|\pZ^1(\sQ,a)|\leq 1$, or if $\pZ^1(\sQ,a) = \{H_1,H_2\}$ such that $\rho_{\sQ}(H_1\cap H_2\setminus a) = d-2$, then $\sQ$ is a principal extension, and hence Proposition \ref{prop:principalExtension} applies. 
So we may assume that $\pZ^1(\sQ,a) = \{H_1, H_2\}$ such  $\rho_{\sQ}(H_1\cap H_2) < d-2$. Since $\sQ$ is connected, it has no coloops, and therefore there is a basis of $\sQ$ that does not contain $a$.  After applying a suitable isomorphism, we may assume that $E=[n]$, $a=n$,  $[d]$ is a basis of $\sQ$, $[d-1]\subset H_1$, and $d-1\notin H_2$. 
    
Let $A$ be the matrix from Formula \eqref{eq:A}. As in the proof of Proposition \ref{prop:principalExtension}, let $A_{\lambda}(\sQ) = \pi_{\sQ}(A_{\lambda})$ where $\pi_{\sQ}$ is the ring map from Formula \eqref{eq:piQ}. Because $[d-1] \cup n \subset H_1$, it is not a basis of $\sQ$, so $\pi_{\sQ}(y_d) = 0$. 
If $1\leq k \leq d-1$, then $\pi_{\sQ}(y_k) = y_k$ because $[d]\setminus k \cup n$ is a basis of $\sQ$. Indeed, if  $ [d]\setminus k \cup n$ is not a basis of $\sQ$, then is closure is a hyperplane in $\pZ^{1}(\sQ,a)\setminus \{H_1\}$, so it must be $H_2$, and hence $[d-1]\setminus k \subset H_1\cap H_2$. This contradicts the hypothesis $\rho_{\sQ}(H_1\cap H_2) < d-2$.

Fix a $(d-1)$--element independent set $\mu\subset H_2\setminus n$. By Lemma \ref{lem:expand}  we have
\begin{equation*}
A_{\mu \cup n}(\sQ) = \sum_{b=1}^{d-1} A_{\mu \cup b}(\sQ)\, y_b \equiv 0 \mod I_{\sQ}.
\end{equation*}
Since $d-1\notin H_2$, the set $\mu \cup (d-1)$ is a basis of $\sQ$.
Solving for $y_{d-1}$ yields the following rational function 
\begin{equation*}
    g = \frac{-1}{A_{\mu \cup (d-1)}(\sQ)} \left( 
    \sum_{b=1}^{d-2} A_{\mu \cup b}(\sQ) \, y_{b} 
    \right).
\end{equation*}
\noindent Let $U$ be the multiplicative semigroup of $S_{\sQ'} \otimes \C[y_1^{\pm},\ldots, y_{d-1}^{\pm}]$ defined by 
\begin{equation*}
    U = \langle f(x_{ij},y_1,\ldots,y_{d-2},g)\, : \, f\in U_{\sQ} \rangle_{\smgp}.
\end{equation*}    
Let $\sQ' = \sQ_{|[n-1]}$.  The semigroup $U_{\sQ'}$ is contained in $U$. Define a ring homomorphism
\begin{gather*}
    \phi: U_{\sQ}^{-1}B_{\sQ} \to U^{-1}(B_{\sQ'} \otimes \C[y_{1}^{\pm}, \ldots, y_{d-2}^{\pm}]) \\
    \phi(x_{ij}) = x_{ij}, \hspace{10pt} \phi(y_1) = y_1, \ldots,  \phi(y_{d-2}) = y_{d-2},  \hspace{10pt} \phi(y_{d-1}) = g.
\end{gather*}
\noindent Now let $\lambda$ be any $(d-1)$--element subset of $H_2\setminus n$. Using Lemma \ref{lem:expand} and the formula $\phi(y_{d-1}) = g$, we have
\begin{equation*}
\phi(A_{\lambda\cup n}(\sQ)) 
= \frac{1}{A_{\mu \cup (d-1)}(\sQ)} 
\sum_{b=1}^{d-2} (
A_{\lambda \cup b}(\sQ) A_{\mu\cup (d-1)}(\sQ)  
- A_{\lambda\cup (d-1)}(\sQ) A_{\mu \cup b}(\sQ) )y_{b}
\end{equation*}
By the quadratic Pl\"ucker relations as in \eqref{eq:pluecker-ideal-gens},  we have
\begin{equation*}
    A_{\lambda \cup b}(\sQ)A_{\mu\cup (d-1)}(\sQ) - A_{\lambda \cup (d-1)}(\sQ) A_{\mu \cup b}(\sQ) \equiv 
    \sum_{m\in \mu} \pm  A_{\lambda \cup m}(\sQ) A_{\mu \cup (b,d-1)\setminus m}(\sQ) \mod I_{\sQ}.
\end{equation*}
Each  $\lambda \cup m$ for $m\in \mu$ is contained in $H_{2}\setminus n$, and hence not a basis of $\sQ'$. This means that the expression on the right lies in $I_{\sQ'}$.  We conclude that $\phi(A_{\lambda \cup n}(\sQ))$ lies in the extension of $I_{\sQ'}$ to $U^{-1}(B_{\sQ'} \otimes \C[y_1^{\pm},\ldots,y_{d-2}^{\pm}])$. Therefore, the morphism $\phi$ descends to a morphism on ring quotients
    \begin{equation*}
        \phi:S_{\sQ} \to U^{-1}(S_{\sQ'} \otimes \C[y_1^{\pm},\ldots,y_{d-2}^{\pm}])
    \end{equation*}
The map $\phi$ is a partial inverse to the inclusion morphism $S_{\sQ'}\otimes \C[y_1^{\pm},\ldots,y_{d-2}^{\pm}] \to S_{\sQ}$. This proves that $\Gr(\sQ;\C) \to \Gr(\sQ';\C)$ factors as an open immersion $\Gr(\sQ;\C)\hookrightarrow \Gr(\sQ';\C)\times (\G_m)^{d-2}$ followed by the projection $\Gr(\sQ';\C)\times (\G_m)^{d-2} \to \Gr(\sQ';\C)$, which is smooth.  
\end{proof}

\begin{corollary}\label{cor:principle_extension_smooth}
If $\sQ$ and $\sQ'$ are as in Proposition \ref{prop:principalExtension}, \ref{prop:coextension} or \ref{prop:2planes}, then  $\Gr(\sQ;\C)$ is smooth (resp. irreducible) if and only if $\Gr(\sQ';\C)$ is smooth (resp. irreducible).
\end{corollary}

\begin{figure}
    \centering
    \includegraphics[width=0.9\textwidth]{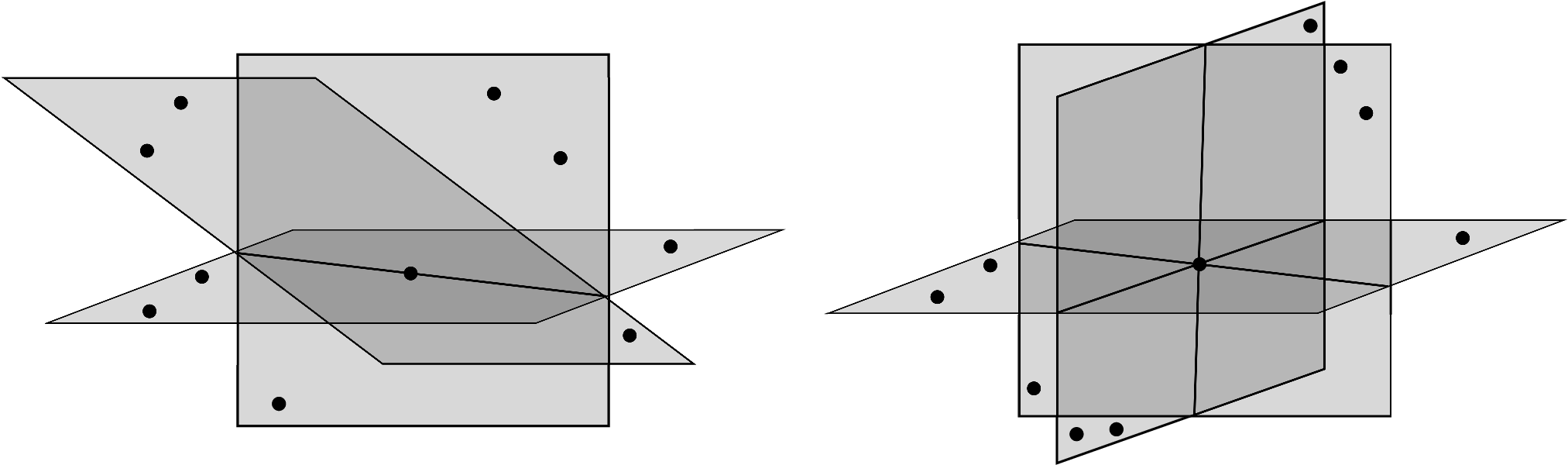}
    \caption{A special and a general configuration of points in $\P^{3}$ realizing the matroid from Example \ref{ex:morphismNotSmooth}}
    \label{fig:notSmoothMorphism}
\end{figure}

\begin{example}
\label{ex:morphismNotSmooth}
    Let $\sQ$ be the $(4,10)$--matroid whose bases are $\binom{[10]}{4}$ except for $H_1 = \{1,2,3,10\}$, $H_2 = \{4,5,6,10\}$, and $H_3 = \{7,8,9,10\}$, and let $\sQ' = \sQ\setminus \{10\}$; in this case $\pB(\sQ') = \binom{[9]}{4}$, i.e., $\sQ'$ is the uniform matroid on $[9]$, see \S \ref{sec:matroid_subdivisions} below.  We show that the morphism $\Gr(\sQ;\C) \to \Gr(\sQ';\C)$ is not smooth. Define the matrix of variables
    \begin{equation*}
    A = \begin{bmatrix}
            1 & 0 & 0 & 0 & x_{11} & x_{12} & x_{13} & x_{14} & x_{15} & y_{1} \\ 
            0 & 1 & 0 & 0 & x_{21} & x_{22} & x_{23} & x_{24} & x_{25} & y_{2} \\ 
            0 & 0 & 1 & 0 & x_{31} & x_{32} & x_{33} & x_{34} & x_{35} & y_{3} \\ 
            0 & 0 & 0 & 1 & x_{41} & x_{42} & x_{43} & x_{44} & x_{45} & 0
        \end{bmatrix}.
    \end{equation*}
    Using Lemma \ref{lem:expand}, we may express the ideal $I_{\sQ}$ as 
    \begin{align*}
    I_{\sQ} = \langle 
        A_{1456} \, y_1 + A_{2456} \, y_2 + A_{3456} \, y_3, \;
        A_{1789} \, y_1 + A_{2789} \, y_2 + A_{3789} \, y_3\rangle.
    \end{align*}
    Solving for $y_{3}$ using the first generator and performing this substitution into the second generator yields
    \begin{equation*}
        (A_{1789}A_{3456} -A_{1456}A_{3789}) \, \frac{y_1}{A_{3456}} + (A_{2789}A_{3456} -  A_{2456}A_{3789}) \,\frac{y_2}{A_{3456}} \in U_{\sQ}^{-1}I_{\sQ}.
    \end{equation*}
    Over the open set
    \begin{equation*}
    \{p_{1789}p_{3456} - p_{1456}p_{3789} \neq 0\}  \cup \{p_{2789}p_{3456} - p_{3789}p_{2456} \neq 0 \} 
    \end{equation*}
    of $\Gr(\sQ';\C)$ the morphism $\varphi_{\sQ,\sQ'}:\Gr(\sQ;\C) \to \Gr(\sQ';\C)$ has 1-dimensional fibers (since we may then eliminate either $y_1$ or $y_2$). Let $V \in \Gr(4,9)$ be the row span of the matrix
    \begin{equation*}
    B = \begin{bmatrix}
1& 0& 0& 0& 1& 2& -2& 3& 5 \\
0& 1& 0& 0& 1& 3& 4& -4& -5 \\
0& 0& 1& 0& 1& 4& 7& -14& -18 \\
0& 0& 0& 1& 1& 1& 1& 1& 1 
    \end{bmatrix}.
    \end{equation*}
    One readily verifies that all maximal minors of this matrix are nonzero, and hence $V\in \Gr(\sQ';\C)$. However, we have
    \begin{equation*}
        B_{1789}B_{3456} -B_{1456}B_{3789} = B_{2789}B_{3456} -  B_{2456}B_{3789} =  0
    \end{equation*}
    so the fiber of $\varphi$ over $V$ is 2-dimensional.
    Therefore, $\Gr(\sQ;\C) \to \Gr(\sQ';\C)$ is not flat, in particular, it is not smooth.  We can understand this geometrically in the following way. For a general point configuration in $\P^{3}$ realizing $\sQ$, the hyperplanes spanned by $H_1,H_2,$ and $H_3$ intersect at a single point. This is the right-side illustrated in Figure \ref{fig:notSmoothMorphism}. However, the hyperplanes are allowed to intersect along a line, this is the left-side illustration in this figure. 
\end{example}

\section{Rank 3 matroids}\label{sec:3n}

Let $\sQ$ be a simple $(3,n)$--matroid. A rank 2 flat of $\sQ$ with at least 3 elements is called a \textit{line}, and we denote the set of lines by $\pL(\sQ)$. As the remaining rank 2 flats are the 2-element subsets not contained in a line, $\sQ$ is uniquely determined by $\pL(\sQ)$.  The matroid $\sQ$ satisfies the \textit{3 lines property} if every element in the ground set is contained in at least 3 lines. The following Proposition is \cite[Lemma~3.2]{NazirYoshinaga} and \cite[Lemma~4.1]{CoreyGrassmannians}; it is a direct consequence of Proposition \ref{prop:2planes}.

\begin{proposition}\label{prop:3lines}
Suppose $\Gr(\sQ;\C)$ is smooth for all $\C$-realizable rank $3$ matroids $\sQ$ on fewer than $n$ elements. 
Then the matroid strata for all $\C$-realizable, $(3,n)$--matroids that are either not simple, not connected, or do not satisfy the 3 lines property are smooth. 
\end{proposition}

\noindent Matroid strata and realization spaces for rank $3$ matroids on $8$ or fewer elements have been studied in \cite{CoreyGrassmannians,CoreyLuber}. We summarize these finding in the following proposition. 
\begin{proposition}
    For $n\leq 8$, the realization space $\pR(\sQ;\C)$ is smooth for all $(3,n)$--matroids realizable over $\C$. They are also irreducible except for the case where $\sQ$ is the M\"{o}bius-Kantor matroid. 
\end{proposition}

\subsection{Rank 3 matroids on 11 or fewer elements}
\label{sec:3leq11}

Here is an outline of our general strategy to show that the realization spaces for $(3,n)$--matroids (for $9\leq n\leq 11$) are smooth. See Examples \ref{ex:3-9} and \ref{ex:3-10} for illustrations.

\begin{itemize}
    \item[Step 1.] Find all simple $(3,n)$--matroids that satisfy the 3 lines property (we modify this step for $n=9$). We use the catalog of small matroids from \texttt{polyDB} \cite{polyDB}, which originates in \cite{MatsumotoMoriyamaImaiBremner}.  
    \item[Step 2.] Remove all matroids not realizable over $\C$. 
    \item[Step 3.]  Using Algorithm \cite[Algorithm~6.12]{CoreyLuber}, systematically eliminate variables to produce a simpler presentation for the coordinate ring of  $\pR(\sQ;\C)$.  Call this new ring $S$. 
    \item[Step 4.] Determine the singular locus, e.g.,  by applying the Jacobian criterion \cite[Corollary~16.20]{Eisenbud} to $S$.
\end{itemize}

\begin{proposition}
\label{prop:3-9smooth}
    All realization spaces for $(3,9)$--matroids are smooth and, except for those listed in Table \ref{tab:39}, irreducible. 
\end{proposition}

\begin{proof}
    Up to $\Sn{9}$--symmetry, there are $383$ simple $(3,9)$--matroids, and $370$ are $\C$-realizable.  Since there are not too many such matroids, we do not isolate those that satisfy the 3 lines property; this allows us to classify those realization spaces that are also irreducible. 
    
    For all matroids except those in Table \ref{tab:39}, we find a reference circuit so that, after applying Step 3 from above, the ideal reduces to $\langle 0 \rangle$. Therefore, these realization spaces are smooth and irreducible. For the matroids in Table \ref{tab:39}, the presentation of the coordinate ring of $\pR(\sQ;\C)$ is obtained using the reference circuit $\{1,2,3,4\}$ and applying Step 3. As the ideals are each generated by a single degree 2 univariate polynomial, these realization spaces are smooth with 2 irreducible components. 
\end{proof}

In Table \ref{tab:39}, the 3rd matroid is the Hesse matroid,  the 4th matroid is studied in Example \ref{ex:3-9}, and the 6th one is the Perles matroid. The remaining matroids have the M\"{o}bius-Kantor matroid as a restriction. 

\begin{table*}[h]
\centering
	\begin{tabular}{ |c|c|c|c| }
	\hline
	 \multirow{2}{5cm}{\centering \textbf{Lines} $\pL(\sQ)$} & \textbf{Ambient} & \multirow{2}{1.5cm}{\centering \textbf{Ideal}} & \multirow{2}{5cm}{\centering \textbf{Semigroup}} \\
   & \textbf{ring} & &\\
\hline 
   \multirow{4}{5cm}{\centering 127, 138, 145, 246, 258, 347, 356, 678} & \multirow{4}{1.5cm}{\centering $\C[x,y,z]$} & \multirow{4}{2cm}{\centering $x^2-x+1$} & \multirow{4}{5.5cm}{\centering {\footnotesize $x, y, z, x - 1, y - 1, z - 1, x - y, xz - x + y - z,  x - y + z, y - z, x - y + z - 1, x + z - 1, xy - y + z, xy - xz + z,  xz - y$}} \\
    & & & \\
    & & & \\
    & & & \\
	\hline 
    \multirow{2}{5cm}{\centering 128, 135, 147, 239, 245, 267, 346, 378, 568} & \multirow{2}{1.5cm}{\centering $\C[x,y]$} & \multirow{2}{2cm}{\centering $x^2-x+1$} & \multirow{2}{5.6cm}{\centering {\footnotesize $x, y, x - 1, y - 1, x - y, x - y - 1, xy + 1, xy - x + 1, xy - y + 1, xy + x - y$}} \\
    & & & \\
	\hline 
   \multirow{2}{5cm}{\centering 127, 138, 145, 169, 239, 246, 258,  347, 356, 489, 579, 678} & \multirow{2}{1.5cm}{\centering $\C[x]$} & \multirow{2}{2cm}{\centering $x^2-x+1$} & \multirow{2}{5cm}{\centering $x, x-1$} \\
     & & & \\
	\hline 
   \multirow{2}{5cm}{\centering 125, 139, 147, 168, 237, 246, 289, 345, 578, 679} & \multirow{2}{1.5cm}{\centering $\C[x]$} & \multirow{2}{2.4cm}{\centering $x^2+1$} & \multirow{2}{5.7cm}{\centering $x,x-1,x+1$} \\
    & & & \\
			\hline 
   \multirow{2}{5cm}{\centering   1258, 136, 149, 237, 269, 345, 467, 579} & \multirow{2}{1.5cm}{\centering $\C[x,y]$} & \multirow{2}{2cm}{\centering $x^2-x+1$} & \multirow{2}{5.4cm}{\centering  $x, y, x-1, y-1, x - y, xy - 1, xy - x - y, x + y - 1$} \\
    & & & \\
			\hline 
   \multirow{2}{5cm}{\centering 1258, 136, 179, 237, 249, 345, 389, 468, 567} & \multirow{2}{1.5cm}{\centering $\C[x]$} & \multirow{2}{2cm}{\centering $x^2+x-1$} & \multirow{2}{4cm}{\centering $x, x+1, x-1$} \\
     & & & \\
			\hline 
   \multirow{2}{5cm}{\centering 1258, 136, 237, 269, 345, 389, 468, 479, 567} & \multirow{2}{1.5cm}{\centering $ \C[x]$} & \multirow{2}{2cm}{\centering $x^2+x+1$} & \multirow{2}{4cm}{\centering $x, x+1, x-1$} \\
     & & & \\
			\hline 
   \multirow{2}{5cm}{\centering 1259, 1367, 238, 247, 345, 469, 568, 789} & \multirow{2}{1.5cm}{\centering $\C[x]$} & \multirow{2}{2cm}{\centering $x^2-x+1$} & \multirow{2}{4cm}{\centering $x, x-1, x-2$} \\
     & & & \\
			\hline 
		\end{tabular}
		\caption{$\C$-realizable $(3,9)$--matroids with reducible realization spaces}
		\label{tab:39}
	\end{table*}

\begin{proposition}
\label{prop:3-10-11smooth}
    All realization spaces for $(3,10)$ and $(3,11)$--matroids are smooth. 
\end{proposition}

\begin{proof}
    Up to $\Sn{10}$--symmetry, there are $5249$ simple $(3,10)$--matroids, $151$ satisfy the 3 lines property, and 107 of these are $\C$-realizable. Similarly, up to $\Sn{11}$--symmetry, there are $232\,928$ simple $(3,11)$--matroids, $16\,234$ satisfy the 3 lines property, and $11\,516$ of them are $\C$-realizable.  For all of these matroids (for both $n=10,11$), we find a reference circuit so that, after applying Step 3, the ideal $I_{\sQ}$ reduces to a principal ideal. Given a ring of the form $S = U^{-1}\C[x_1,\ldots,x_m]/\langle f\rangle$, the ideal of the singular locus of $\Spec(S)$ is 
    \begin{equation*}
        J = \langle f, \frac{\partial f}{\partial x_1}, \ldots, \frac{\partial f}{\partial x_m}\rangle.
    \end{equation*}
    Thus $\Spec(S)$ is smooth if and only if $J = \langle 1 \rangle$ in $U^{-1}\C[x_1,\ldots,x_m]$, equivalently, the saturation of $J$ by $U$ is the unit ideal in $\C[x_1,\ldots,x_m]$. We use this to verify that these realization spaces are smooth. 
\end{proof}

\noindent Theorem \ref{thm:intro-3-leq11-smooth} now follows from Propositions \ref{prop:3-9smooth} and \ref{prop:3-10-11smooth}.

\begin{example}
\label{ex:3-10}
    Let $\sQ$ be the simple $(3,10)$--matroid whose lines are
    \begin{equation*}
        \mathcal{L}(\mathsf{Q}) = \left\{
        \begin{array}{c}
        \{1,2,5\}, \{1,3,6\}, \{1,4,8\}, \{2,3,7\}, \{2,4,9\}, \{2,6,10\}, \\ 
        \{3,4,5\}, \{4,6,7\}, \{5,9,10\},\{6,8,9\}, \{7,8,10\}
        \end{array}
         \right\}.
    \end{equation*}
    Let $B=\C[x,y]$ and define the $B$-valued matrix $A$ by
    \begin{equation*}
    A = 
\begin{bmatrix}
    1 & 0 & 0 & 1 & 1 & x & 0 & x^2-xy-1 & 1 & x \\
    0 & 1 & 0 & 1 & 1 & 0 & x & x-y-1 & -x+y+1 & y\\
    0 & 0 & 1 & 1 & 0 & 1 & x-1 & x-y-1 & 1 & 1
\end{bmatrix}.
    \end{equation*}
    Define the ideal $I\subset \C[x,y]$ and multiplicative semigroup $U \subset \C[x,y]$ by
\begin{align*}
    I &= \langle x^2y - x^2 - xy^2 + xy - y \rangle, \\
    U &= \left\langle
    \begin{array}{c}
x, y, x - 1, y - 1, x - y,  x - y - 1, xy - x + 1, xy - x - y, x^2 - xy - 1, \\
x^2 - xy + y, x^2 - xy - x + y + 1, x^3 - 2x^2 - xy^2 + 2xy - 2y 
    \end{array}
    \right\rangle_{\smgp}.
\end{align*}
Denote by $f$ the unique generator of $I$. 
The $\C$-realizations of $\sQ$ are exactly the row spans of $A$ for any $(x,y)$ such that $f(x,y)=0$ and $g(x,y) \neq 0$ for all $g\in U$. The singular locus of $I$ is the ideal  $J\subset U^{-1}\C[x,y]$ defined by
\begin{align*}
    J = \langle f,\, \frac{\partial f}{\partial x},\, \frac{\partial f}{\partial y}  \rangle 
    = \langle x^2y-x^2-xy^2+xy-y,\, 2x-y,\, x^2-2xy+x-1\rangle.
\end{align*}
From the second generator, $y\equiv 2x \mod J$. Performing this substitution in the remaining two generators yields $-2x^3+x^2-2x$ and $-3x^2+x-1$. Because these univariate polynomials are relatively prime, the ideal $J$ is the unit ideal, and therefore $\pR(\sQ;\C)$ is smooth. 
\end{example}

\subsection{(3,12)--matroids and beyond}
\label{sec:3-12}
In this section we exhibit a rank $3$ matroid on $12$ elements, denoted $\sQ_{\sing}$, whose realization space is singular. Using this, we show that singular realization spaces exist for $(3,n)$-matroids with $n>12$.
\begin{remark}
We discovered the matroid $\sQ_{\sing}$ by experiments using \texttt{OSCAR}. While our examination of $(3,12)$-matroids is not exhaustive, our search yields $76$ singular realization spaces. This data is available in the github repository, see the Code section of the introduction. 
\end{remark}

Define
\begin{align}
\label{eq:singularMatroid}
    \mathcal{L} = 
    \left\{ \begin{array}{l} 
\{1,2,6,8\}, \{1,3,5,7\}, \{1,9,12\}, \{2,4,5,9\}, \{2,7,11\}, \{3,4,6\}, \\ 
\{3,8,9\}, \{3,10,12\}, \{4,7,8\},  \{4,10,11\}, \{5,6,10\}, \{8,11,12\}
\end{array} \right\}
\end{align}
and let $\sQ_{\sing}$ be the simple $(3,12)$--matroid with $\pL(\sQ_{\sing}) = \pL$. A line configuration realizing $\sQ_{\sing}$ is displayed in Figure \ref{fig:singular_configuration}.

\begin{figure}[h]
    \centering
    \includegraphics[width = .4\textwidth]
    {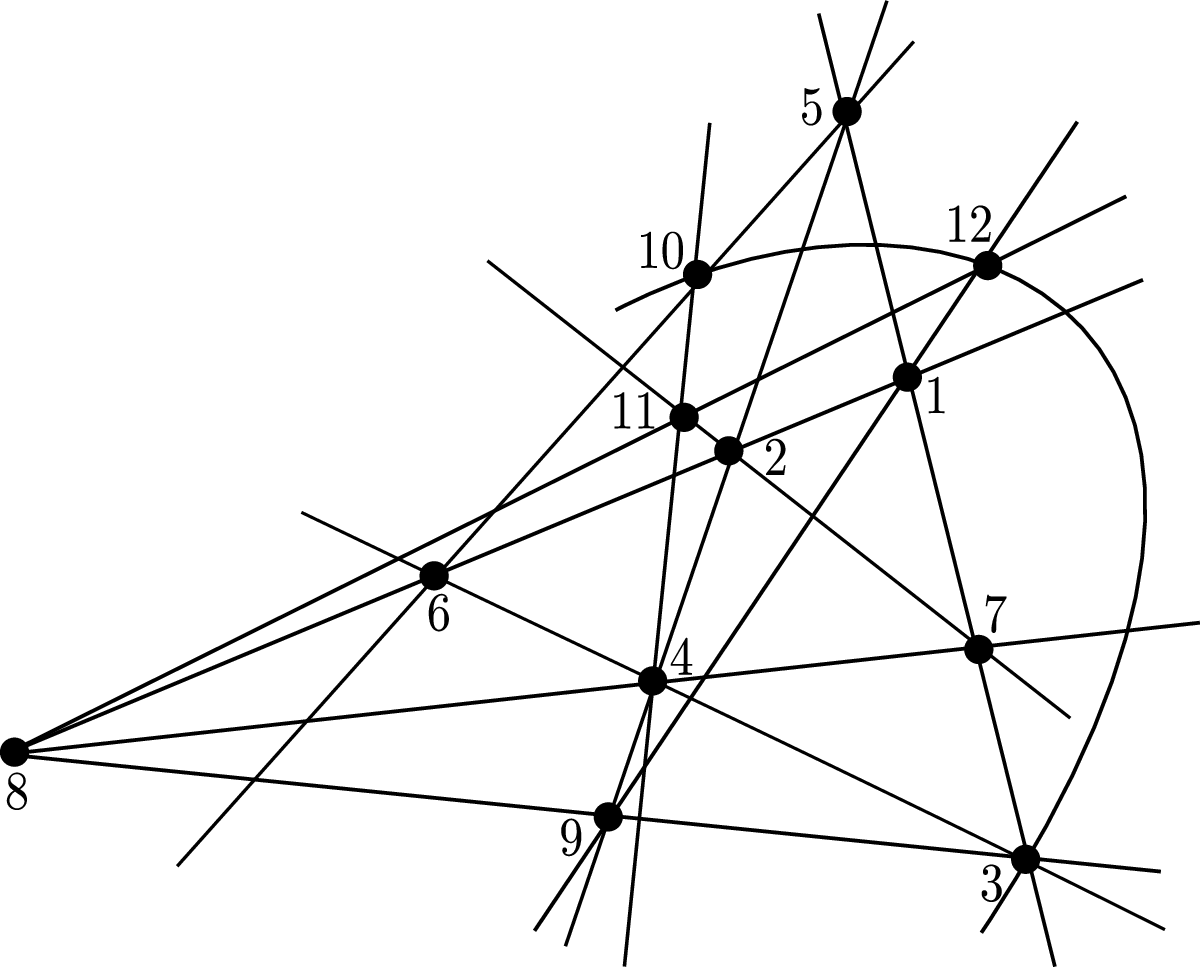}
    \caption{Projective realization of $\sQ_{\sing}$}
    \label{fig:singular_configuration}
\end{figure}

Let $B = \C[x,y]$ and define the $B$-valued matrix 
\begin{equation}
\label{eq:representationMatrix312}
    A = \begin{bmatrix}
    1 & 0 & 0 & 1 & 1 & 1 & y-1 & 1 & 1 & x & y-1 & x \\
    0 & 1 & 0 & 1 & 0 & 1 & 0 & y & y & y & -xy^2+2y^2-y & y \\
    0 & 0 & 1 & 1 & 1 & 0 & y & 0 & 1 & x-y & y & 1
    \end{bmatrix}.
\end{equation}
The coordinate ring of $\pR(\sQ_{\sing};\C)$ is isomorphic to $U^{-1}B/I$ where 
\begin{equation*}
    I = \langle(xy+x-2y)(y^2-y+1)\rangle
\end{equation*}
and $U$ is generated by
\begin{gather*}
x,
y,
x - 1,
x - 2,
y - 1,
y + 1,
x - y,
x - 2y,
x - y - 1,
xy - y + 1,
xy - 2y + 1, \\
x + y^2 - y,
xy - 2y + 2,
x + y^2 - 2y,
x + y^2 - y - 1,
xy - y^2 - y + 1, \\
xy^2 - y^2 + y - 1,
xy^2 - 2y^2 + y - 1,
xy^2 - 2y^2 + 2y - 1,
x^2y - xy^2 - 2xy + x + 2y^2.
\end{gather*}
\begin{theorem}
\label{thm:singular3-12}
    The realization space $\pR(\sQ_{\sing};\C)$ has 3 irreducible components and 2 nodal singularities. 
\end{theorem}

\begin{proof} 
The description above realizes $\pR(\sQ_{\sing};\C)$ as a dense open subvariety of an affine plane curve with 3 irreducible components: $X_1$ corresponds to $xy + x - 2y = 0$, $X_2$ corresponds to $y-\omega=0$, and $X_{3}$ corresponds to $y-\overline{\omega}=0$; here $\omega$ and $\overline{\omega}$  are the solutions to  $y^2-y+1=0$. Individually, these irreducible components are smooth, so the singularities of $\pR(\sQ_{\sing};\C)$ are the intersection points of these curves. Solving the system of equations $xy+x-2y = y^2-y+1 = 0$ yields the following two points
    \begin{equation*}
    \sq_1 = \left( \frac{3-\sqrt{-3}}{3},\frac{1-\sqrt{-3}}{2}\right) \hspace{20pt}
    \sq_2 = \left(\frac{3+\sqrt{-3}}{3},\frac{1+\sqrt{-3}}{2}\right).
    \end{equation*}
  It remains to show that $f(\sq_1)$ and $f(\sq_2)$ are nonzero for all $f$ in the generating set of $U$ recorded above, which is a routine verification. 
  Since $X_{1}$ intersects $X_{2}$ and $X_{3}$ transversely at these points, these singularities are nodes.
\end{proof}

Let $3\leq d \leq \frac{n}{2}$ and  $(d,n)$ is not any of the following pairs:
\begin{equation}
    \label{eq:openPairs}
\begin{gathered}
    (3,6), (3,7), (3,8), (3,9), (3,10), (3,11), (4,8), (4,9), (4,10), (4,11), (4,12), \\
    (5,10), (5,11), (5,12), (5,13), (6,12), (6,13), (6,14), (7,14), (7,15), (8,16).
\end{gathered}
\end{equation}
For such a pair, denote by $\sQ_{d,n,\sing}$ the matroid obtained from $\sQ_{\sing}$ from Theorem \ref{thm:singular3-12} by $d-3$ free coextensions and $n-(d+9)$ free extensions (here, \textit{free} means that we (co)extend into the full ground set). This is a matroid of rank $d$ on $[n]$.    The following is a reformulation of Theorem \ref{thm:intro-3-geq12-singular} in the introduction. It is a direct consequence of Propositions \ref{prop:principalExtension}, \ref{prop:coextension}, and Theorem \ref{thm:singular3-12}.

\begin{theorem}
\label{thm:singular3-n}
For $(d,n)$ as above, the realization space $\pR(\sQ_{d,n,\sing};\C)$ has nodal singularities.  
\end{theorem}

\noindent Combining this result with Theorems \ref{thm:intro-3-leq11-smooth} and \ref{thm:intro-4-n-matroids-smooth}, we see that there are only 13 pairs $(d,n)$ with $3\leq d \leq \frac{n}{2}$ for which we do not know if there is a  $\C$-realizable $(d,n)$--matroids whose realization space is singular.

\section{Rank 4 matroids}\label{sec:4n}

For a simple and connected $(4,n)$-matroid $\sQ$, a hyperplane with at least 5 elements is called a \textit{plane}, and we denote the set of planes by $\pL(\sQ)$.  We say that a simple and connected $(4,n)$--matroid $\sQ$ satisfies the \textit{3 planes property} if every $a\in [n]$ is contained in at least $3$ planes of $\sQ$. Similarly, $\sQ$ satisfies the \textit{4 planes property} if every $a\in [n]$ is contained in at least $4$ planes of $\sQ$. The following is a direct consequence of Propositions \ref{prop:directsum}, \ref{prop:principalExtension}, and \ref{prop:2planes}.

\begin{proposition}
\label{prop:4planes}
Let $\sQ$ be a connected and $\C$-realizable  $(4,n)$--matroid such that
\begin{enumerate}
    \item $\sQ$ is not simple,
    \item $\sQ$  does not satisfy the 3 planes property, or
    \item $n\leq 9$ and $\sQ$ does not satisfy the 4 planes property. 
\end{enumerate}
    If the matroid strata are smooth for all $\C$-realizable rank $4$ matroids on $<n$ elements, then $\Gr(\sQ;\C)$ is smooth. Similarly, if the matroid strata are smooth and irreducible for all $\C$-realizable rank $4$ matroids on $<n$ elements, then $\Gr(\sQ;\C)$ is smooth and irreducible. 
\end{proposition}

\begin{proposition}
    If $\sQ$ is a $\C$--realizable $(4,n)$--matroid for $n\leq 7$, then $\Gr(\sQ;\C)$ is smooth and irreducible.
\end{proposition}

\begin{proof}
    This follows from duality (Proposition \ref{prop:duality}) and the fact that the matroid strata for matroids of rank at most $3$ on at most  $7$ elements are smooth and irreducible.  
\end{proof}

\begin{proposition}
\label{prop:4-8-smooth}
    All realization spaces for $\C$-realizable $(4,8)$-matroids are smooth and, except for those listed in Table \ref{tab:48}, irreducible. 
\end{proposition}

\begin{proof} 
We use \texttt{OSCAR} to prove this proposition. 
Up to $\Sn{8}$--symmetry, there are $592$ simple and connected $(4,8)$--matroids. Of those, $92$ satisfy the $4$ planes property and $66$ are $\C$-realizable. (Those matroids that are not simple, not connected, or do not satisfy the $4$ planes property have smooth and irreducible realization spaces by Proposition \ref{prop:4planes}).
For $63$ of these, we find a reference circuit such that $I_{\sQ}$ reduces to $\langle 0 \rangle$ in Step 3 from \S \ref{sec:3leq11}. So these realization spaces are smooth and irreducible.  The remaining $3$ are listed in Table \ref{tab:48}. The presentations listed in this table are obtained by using $\{1,2,3,4,5\}$ as a reference circuit and applying Step 3 from \S\ref{sec:3leq11}. As these ideals are each generated by a single degree 2 univariate polynomial, these realization spaces are smooth with 2 irreducible components. 
\end{proof}

\begin{table}[h]\label{table:4-8}
\centering
	\begin{tabular}{ |c|c|c|c| }
	\hline
	 \multirow{2}{5cm}{\centering \textbf{Planes} $\pL(\sQ)$} & \textbf{Ambient} & \multirow{2}{1.5cm}{\centering \textbf{Ideal}} & \multirow{2}{3cm}{\centering \textbf{Semigroup}} \\
   & \textbf{ring} & &\\
\hline 
   \multirow{2}{6cm}{\centering 3467, 2567, 2458, 2378, 1568, 1357, 1348, 1247, 1236} & \multirow{2}{1.5cm}{\centering $\C[x]$} & \multirow{2}{3cm}{\centering $x^2-3x+1$} & \multirow{2}{3cm}{\centering { $x, x-2,$ \\ $ x-1, x-3$}} \\
    & & & \\
	\hline
  \multirow{2}{6cm}{\centering 4568, 3467, 2567, 2378, 1357,
  1348, 1258, 1247, 1236} & \multirow{2}{1.5cm}{\centering $\C[x]$} & \multirow{2}{3cm}{\centering $3x^2-3x+1$} & \multirow{2}{4cm}{\centering { $x,x-1,3x-1,$ \\ $3x-2,2x-1, x-3$}} \\
    & & & \\
	\hline
 \multirow{2}{6cm}{\centering 12367, 5678, 3456, 2478, 2358, 1457, 1248, 1268, 1256, 1246} & \multirow{2}{1.5cm}{\centering $\C[x]$} & \multirow{2}{3cm}{\centering $3
x^2-x+1$} & \multirow{2}{3cm}{\centering {$x,x-1$}} \\
    & & & \\
	\hline 
	\end{tabular}
	\caption{$\C$-realizable $(4,8)$--matroids with disconnected realization spaces}
	\label{tab:48}
\end{table}

\noindent Interestingly, there is exactly one simple and connected $(4,8)$--matroid satisfying the $4$ planes property that does \emph{not} have a circuit of size $5$. 

\begin{example}\label{example:weird 4-8}
Let $\sQ$ be the matroid of the 8 points in $(\F_2)^{3}$, i.e., the matroid of the $\F_{2}$--valued matrix
\begin{equation*}
    A = \begin{bmatrix}
       0 &1 & 0 &0 &1 &1 &0 &1\\
       0 &0 &1 &0 &1 &0 &1 &1\\
       0 &0 &0 &1 &0 &1 &1 &1\\
        1 &1 &1 &1 &1 &1 &1 &1
     \end{bmatrix}.
\end{equation*}
All circuits of this matroid have size at most $4$. While realizable over $\F_2$, this matroid is not realizable over $\C$. 
\end{example}

\begin{proposition}
\label{prop:4-9-smooth}
    All realization spaces for $\C$-realizable $(4,9)$--matroids are smooth.
\end{proposition}
\begin{proof}
    Up to $\Sn{9}$--symmetry, there are $185\,911$ simple and connected matroids $(4,9)$--matroids, $61\, 228$  satisfy the $4$ planes property, and $39\,246$ of these are $\C$-realizable. 
    For each of these matroids, we find a reference circuit so that, after applying Step 3 from \S\ref{sec:3leq11}, the ideal $I_{\sQ}$ reduces to a principal ideal. We verify that these spaces are smooth using the Jacobian criterion as outlined in the proof of Proposition \ref{prop:3-10-11smooth}. 
\end{proof}
\section{Flag varieties}\label{flag_varieties}
The \emph{complete flag variety} $\Fl(n;\F)$ paramaterizes sequences (also referred to as \emph{flags}) of nested linear subspaces $V_{\bullet} \coloneqq V_1\subset V_2\subset\dots\subset V_n = \F^n$, such that $\dim(V_d) = d$. From $V_{\bullet}$ we obtain a sequence of matroids on $[n]$, $\sQ_{\bullet} = (\sQ_{1},\dots,\sQ_{n})$, where $\sQ_{d} = \sQ(V_{d})$. The sequence $\sQ_{\bullet}$ is known as an \emph{$\F$-realizable complete flag matroid} on $[n]$, and $V_{\bullet}$ is an \emph{$\F$-realization} of $\sQ_{\bullet}$. In either $V_{\bullet}$ or $\sQ_{\bullet}$, the  $d$-th element of the flag is called the \emph{$d$-th constituent}.  The \emph{bases} of $\sQ_{\bullet}$ are simply $\mathcal{B}(\sQ_{\bullet}) = \cup^{n}_{i=1}\mathcal{B}(\sQ_i)$. Similar to usual matroids, flag matroids capture the combinatorics of nested linear subspaces. 

\begin{remark}\label{rem:flag_matroids}
More general flag varieties exist, where we allow constituents to skip dimensions. Similarly, flag matroids on $[n]$ need not have $n$ constituents. Like usual matroids, flag matroids admit several cryptomorphic definitions which are strictly combinatorial. For example, a sequence of matroids $(\sQ_{i},\dots,\sQ_{k})$ is a flag matroid if the set of flats of $\sQ_{i}$ is a subset of the flats of $\sQ_{i+1}$. Not all flag matroids are realizations of flags of vector spaces. Furthermore, there exist non-realizable examples where each constituent matroid of the flag is realizable. For a  comprehensive treatment of flag matroids, see \cite[Chapter 2]{BorovikGelfandWhite:2003}. 
\end{remark}

Analogous to the Grassmannian, the flag variety $\Fl(n;\F)$ admits a decomposition into strata indexed by realizable complete flag matroids on $[n]$. Denote by $\Fl(\sQ_{\bullet};\F)$ the subvariety of $\Fl(n;\F)$ whose closed points correspond to realizations of $\sQ_{\bullet}$. A flag $V_{\bullet} = V_{1}\subset\dots\subset V_{n} = \F^{n}$ may be recorded as an $n\times n$ $\F$-matrix $A_{V_{\bullet}}$, where $V_{d}$ is the span of the first $d$ rows of $A_{V_{\bullet}}$. Hence, we embed $\Fl(n;\F)$ into $\P_{\F}^{n-1}\times \P^{\binom{n}{2}-1}_{\F}\times\dots\times \P_{\F}^{n-1}$ by concatenating the Pl\"ucker coordinates of the constituents of each flag.

\begin{proposition}\label{prop:flag_iso}
    Let $\sQ$ be a rank $d$ matroid on $[n]$, realizable over $\F$. Then there exists a complete flag matroid $\sQ_{\bullet} =(\sQ_1,\dots,\sQ_n)$ such that $\sQ_d=\sQ$, and $\Gr(\sQ;\F)\cong \Fl(\sQ_{\bullet};\F)$.
\end{proposition}
\begin{proof}
    Fix a $(d,n)$-matroid $\sQ$ and let $W\in\Gr(\sQ;\F)$.  As before, we can assume $[d]$ is a basis of $\sQ$. Then $W$ is given by the row span of the $d\times n$ matrix $A_{W} = \begin{bmatrix} I_d &A \end{bmatrix}$. We extend $W$ to a full flag $W_{\bullet}\in\Fl(n;\F)$ as follows.
    
    Define the $n\times n$ $\F$-matrix \begin{equation}\label{eq:flag_realization}
        A_{W_{\bullet}} = \begin{bmatrix}
        I_d &A\\
        0 &I_{n-d}
    \end{bmatrix}.
    \end{equation}
We obtain the full flag $W_{\bullet} = W_1\subset\dots\subset W_n$ where $W_i$ is the span of the first $i$ rows of $A_{W_{\bullet}}$. Then $W_{\bullet}$ is a realization of the flag matroid $\sQ_{\bullet} = (\sQ_1,\dots,\sQ_n)$, where $W_i\in\Gr(\sQ_i;\F)$. Clearly $\sQ_d = \sQ$.

We then have \begin{equation}\label{eq:flag extension}
    \mathcal{B}(\sQ_i) = \begin{cases}
    \mathcal{B}(\sQ/\{i+1,\dots,d\}) \text{ if } i<d\\
    \mathcal{B}(\sQ)\text{ if } i = d\\
     \{\lambda\cup \{d+1,\dots,i\}:\lambda\in\mathcal{B}(\sQ_{d}\setminus  \{d+1,\dots,i\})\} \text{ if } i>d.
\end{cases}
\end{equation}

We can see that the map $\Gr(\sQ;\F)\to\Fl(\sQ_{\bullet};\F)$ given by $A_{W}\mapsto A_{W_{\bullet}}$, where $A_{W_{\bullet}}$ is of the form in Equation \ref{eq:flag_realization}, is injective. To see that it is surjective, let $V_{\bullet}\in\Fl(\sQ_{\bullet};\F)$. By \ref{eq:flag extension}, the \emph{standard flag} $(\{1\},\{1,2\},\dots,\{1,\dots,n\})\in\pB(\sQ_{\bullet})$.  We therefore apply the construction in \cite[\S5.1]{corey_olarte_2022} to understand $A_{V_{\bullet}}$. Consider the entry $a_{i,j}$ of $A_{V_{\bullet}}$. Then we have $a_{i,j} \neq 0$ if and only if $[i-1]\cup j\in\pB(\sQ_{i})$. But for $i>d$, $i$ is a coloop of $\pB(\sQ_{i})$. Therefore we have $a_{i,j} = 0$ for all $j>i>d$. Then the realization matrix $A_{V_{\bullet}}$ is determined by the $d$-th constituent of $V_{\bullet}$, and is of the form in Equation \ref{eq:flag_realization}. Hence the mapping $\Gr(\sQ;\F)\to\Fl(\sQ_{\bullet};\F)$ given by $A_{W}\mapsto A_{W_{\bullet}}$ is a bijection, and the Pl\"ucker embedding induces an isomorphism of affine schemes.
\end{proof}
\noindent The following is immediate.
\begin{corollary}\label{corr:singular_flag_stratum}
    Let $\sQ$ be a rank $d$ matroid on $[n]$, representable over $\F$. Then the flag stratum $\Fl(\sQ_{\bullet};\F)\subset \Fl(n;\F)$, such that $\mathcal{B}(\sQ_{\bullet})$ is given by Equation \ref{eq:flag extension}, is singular (resp. irreducible) if and only if $\Gr(\sQ;\F)$ is singular (resp. irreducible).
\end{corollary}
Hence, if we take $d=3$ and $\sQ_d = \sQ_{3,n,\sing}$ from \S \ref{sec:3-12}, Corollary  \ref{corr:singular_flag_stratum} implies the existence of a singular flag matroid stratum in $\Fl(n;\C)$ for $n\geq 12$.
Since  singular flag matroid strata need not be obtained in this manner, it is reasonable to ask if there exists $n<12$ such that $\Fl(n;\C)$ contains a singular flag matroid stratum.

\section{Singular initial degenerations}

We begin this section with a brief overview of initial degenerations of Grassmannians and their relation to (valuated) matroids as described in \cite{CoreyGrassmannians}. In \S\ref{sec:singularInitialDeg312}, we develop refined techniques to study initial degenerations corresponding to corank vectors of simple rank-3 matroids on the way to proving Theorem \ref{thm:not-shon-geq-12}.

\subsection{Initial degenerations and tropicalization}
\label{sec:initial_degenerations}
Let $S = \C[x^{\pm}_{1},\dots,x^{\pm}_{a}]$, and given $\su = (u_1,\ldots,u_a) \in \Z^{a}$ set $x^{\su} = x_{1}^{u_1}\cdots x_{a}^{u_a}$. Given $f\in S$ and $\sw\in (\R^{a})^{\vee}$, the $\sw$-\emph{initial from} of $f$ is 
\begin{equation*}
\init_{\sw}f = \sum\limits_{\substack{\langle \su,\sw\rangle\\\text{ is minimal}}}c_{\su}x^{\su}\hspace{.5 cm} 
\text{where} 
\hspace{.5 cm} f = \sum c_{\su}x^{\su},
\end{equation*}
\noindent where each $c_{\su}$ in the second sum is nonzero. Similarly, for a variety $X = V(I)$ defined by an ideal $I\subset S$, we obtain 
\begin{equation*}
\init_{\sw} I =\langle \init_{\sw}f\, : \, f\in I\rangle\hspace{.5 cm}\text{ and }\hspace{.5 cm}\init_{\sw}X = V(\init_{\sw}I),    
\end{equation*}
the $\sw$-\emph{initial ideal} of $I$ and $\sw$-\emph{initial degeneration} of $X$ respectively. The \emph{tropicalization of $X$} is 
\begin{equation*}
\Trop\, X = \{\sw\in (\R^{a})^{\vee} \, :\, \init_{\sw}I \neq\langle 1 \rangle\}.
\end{equation*}
When $X = \Gr^{\circ}(d,n)$, we write $\Trop\, \Gr^{\circ}(d,n) = \TGr^{\circ}(d,n)$.  The set $\TGr^{\circ}(d,n)$ is invariant under translation by $\mathbf{1} = (1,\ldots,1)$, so we view $\TGr^{\circ}(d,n) \subset (\R^{\binom{[n]}{d}})^{\vee}/\Zone$. 

\subsection{Matroidal subdivisions of hypersimplices}\label{sec:matroid_subdivisions}
Denote by $\epsilon_1,\ldots,\epsilon_n$ the standard basis of $\Z^{n}$ and $\epsilon_1^*,\ldots,\epsilon_n^{*}$ its dual basis in $(\Z^{n})^{\vee}$. Given $\lambda = \{i_1,\ldots,i_d\} \in \binom{[n]}{d}$, set  $\epsilon_{\lambda} = \epsilon_{i_1} + \cdots + \epsilon_{i_d}$ and $\epsilon_{\lambda}^{*} = \epsilon_{i_1}^{*} + \cdots + \epsilon_{i_d}^{*}$. For a $(d,n)$-matroid $\sQ$, its  \emph{matroid polytope} $\Delta(\sQ)\subset \R^{n}$ is
\begin{equation*}
\Delta(\sQ) = \conv( \epsilon_{\lambda} \, : \, \lambda\in \pB(\sQ)).
\end{equation*}

Given a finite set $E$ and a nonnegative integer $d\leq |E|$, the rank-$d$ \textit{uniform matroid} on $E$, denoted $\sU_{d,E}$, is the matroid whose bases are the $r$-element subsets of $E$. When $E=[n]$, we simply write $\sU_{d,n}$. The matroid polytope of $\sU_{d,n}$ is called the \textit{hypersimplex}, and is denoted by $\Delta(d,n)$. The vertices of $\Delta(d,n)$ are exactly the vectors $\epsilon_{\lambda}$ for $\lambda \in \binom{[n]}{d}$.  

Given $\sw\in (\R^{\binom{[n]}{d}})^{\vee}/\Zone$, the \textit{regular subdivision} of $\Delta(d,n)$ induced by $\sw$, denoted by $\pQ(\sw)$, is obtained by lifting the vertex $\epsilon_{\lambda}$ to height $\sw_{\lambda}$ in $\R^{n}\times \R$, and projecting the lower faces of the convex hull of the resulting point configuration back to $\Delta(d,n)$; for a comprehensive treatment of regular subdivisions, see \cite[Chapter~2]{DeLoeraRambauSantos} or  \cite[Chapter~7]{GelfandKapranovZelevinsky} (where they are called \textit{coherent subdivisions}). By \cite[Proposition~2.2]{Speyer2008}, if $\sw \in \TGr^{\circ}(d,n)$, then $\pQ(\sw)$ is \textit{matroidal}, i.e., the cells of the subdivision are polytopes of $(d,n)$-matroids. We denote these matroids by $\sU_{\sw}^{\sv}$, and their bases are given by  
\begin{equation*}
    \pB(\sU_{\sw}^{\sv}) = \{ \lambda \in \textstyle{\binom{[n]}{d}} \, : \, \bk{\epsilon_{\lambda}}{\sv} + \sw_{\lambda} \text{ is minimal}   \}
\end{equation*}
for $\sv \in (\R^{n})^{\vee}$. 

\subsection{Inverse limits of matroid strata}\label{sec:inverse_limits_tight_spans}
 Let $\sw\in\TGr^{\circ}(d,n)$.   We view $\pQ(\sw)$ as a poset via the face relation. If $\Delta(\sQ')$ is a face of $\Delta(\sQ)$, then the inclusion of polynomial rings $B_{\sQ'} \subset B_{\sQ}$ induces a morphism of matroid strata 
\begin{equation}
\label{eq:face-maps}
    \varphi_{\sQ,\sQ'}:\Gr(\sQ;\F) \to \Gr(\sQ';\F)
\end{equation}  
(provided these matroids are $\F$-realizable), see \cite[Proposition~I.6]{Lafforgue2003} or \cite[Proposition~3.1]{CoreyGrassmannians}. This yields a diagram of matroid strata over the poset $\pQ(\sw)$, denote the corresponding (finite) inverse limit by $\Gr(\sw)$. The following result in \cite{CoreyGrassmannians} is essential in our construction of singular initial degenerations.

\begin{theorem}\label{thm:closed_immersion}
Let $\sw\in\TGr^{\circ}(d,n)$. The inclusions 
\begin{equation*}
    \C[p_{\lambda} \, : \, \lambda \in \pB(\sQ)] \subset \C[p_{\lambda} \, : \, \lambda \in \textstyle{\binom{[n]}{d}}]
\end{equation*}
for $\Delta(\sQ) \in \pQ(\sw)$ induce a closed immersion 
\begin{equation*}
\psi_{\sw}: \init_{\sw}\Gr^{\circ}(d,n)\hookrightarrow\Gr(\sw).
\end{equation*} 
\end{theorem}

\noindent The dual graph of $\pQ(\sw)$, denoted $\Gamma(\sw)$, is the graph that has a vertex $v_{\sQ}$ for each maximal cell $\Delta(\sQ)$ in $\pQ(\sw)$, and $v_{\sQ_1}$ is connected to $v_{\sQ_2}$ by an edge if the polytopes $\Delta(\sQ_1)$ and $\Delta(\sQ_2)$ share a common facet. The graph $\Gamma(\sw)$ is a subposet of $\pQ(\sw)$ with the relation  $e\prec v$ if $v$ is a vertex of the edge $e$. By \cite[Proposition~C.12]{CoreyGrassmannians} (Appendix by Cueto), the inverse limit over this smaller poset is isomorphic to $\Gr(\sw)$.

Determining whether the initial degeneration $\init_{\sw}\Gr^{\circ}(d,n)$ is singular is a challenging computational problem. In contrast, determining whether the finite limit $\Gr(\sw)$ is singular is far more tractable especially when $\Gamma(\sw)$ is a star-shaped tree. This means that the maximal cells of $\pQ(\sw)$ are the matroid polytopes of matroids $\sQ$, $\sL_1$, $\ldots$, $\sL_k$ such that $\Delta(\sL_i)$ and $\Delta(\sL_j)$ do not share a common facet (for $i\neq j$) and  $\Delta(\sQ)$ shares a facet with each $\Delta(\sL_i)$. The facet common to $\Delta(\sQ)$ and $\Delta(\sL_{i})$ is itself a matroid polytope, denote the corresponding matroid by $\sL_i'$. Then the inverse limit $\Gr(\sw)$ is isomorphic to a fiber product
\begin{equation*}
    \Gr(\sw) \cong \Gr(\sQ) \times_{X'} X, \quad \text{where} \quad X = \prod_{i=1}^{k} \Gr(\sL_i) \quad \text{and} \quad X' = \prod_{i=1}^{k} \Gr(\sL_i'). 
\end{equation*}
The morphisms $\Gr(\sQ) \to X'$ and $X\to X'$ are naturally defined using the the face morphisms in \eqref{eq:face-maps}. If $X\to X'$ is smooth and surjective, then its pullback by $\Gr(\sQ) \to X'$ produces a smooth and surjective morphism $\Gr(\sw)\to \Gr(\sQ)$ by \cite[Proposition~7.3.3]{EGAIV}. Thus, to prove that $\Gr(\sw)$ is singular, we need only show that $\Gr(\sQ)$ is singular and that the morphism $X\to X'$ is smooth and surjective. If $\psi_{\sw}$ is an isomorphism (which need not be the case, see \cite[Example~8.2]{CoreyGrassmannians}), then $\init_{\sw}\Gr^{\circ}(d,n)$ is also singular. 

In the next section, we review the notions of corank vector and paving matroid. We show that the regular subdivision of $\Delta(d,n)$ induced by a corank vector $\sw$ of a paving matroid $\sQ$ is matroidal and its dual graph $\Gamma(\sw)$ is a star-shaped tree with $\sQ$ as the matroid corresponding to the central node. Following the above strategy, we prove in Proposition \ref{prop:compareLimitToCenter} that the scheme $\Gr(\sw)$ has the same singularity type as $\Gr(\sQ)$. To prove Theorem \ref{thm:not-shon-geq-12}, we show that the matroids from Theorem \ref{thm:singular3-n}  are paving and, when $\sw$ is the corank vector of one of these specific matroids, we show that $\init_{\sw}\Gr^{\circ}(d,n) \to \Gr(\sw)$ is an isomorphism.

\subsection{Cyclic flats and paving matroids}
\label{sec:pavingMatroids}
Let $\sQ$ be a $(d,n)$--matroid. A \textit{cyclic flat} of $\sQ$ is a flat $\eta$ such that $\sQ_{|\eta}$ has no coloops, equivalently, it is a flat that is a union of circuits \cite[p.76]{Oxley}. A \textit{cyclic hyperplane} is a hyperplane that is a cyclic flat, and the set of cyclic hyperplanes is denoted by $\pZ^{1}(\sQ)$.

The matroid $\sQ$ is \textit{paving} if each circuit of $\sQ$ has at least $d$ elements.  For rank at least $2$, paving matroids must be loopless, and for rank at least $3$, paving matroids must be simple. It is conjectured that asymptotically almost all matroids are paving \cite{CarpoRota, MayhewNewmanWelshWhittle}. In terms of cyclic flats, paving matroids admit the following characterization.
\begin{proposition}
\label{prop:pavingCyclicFlats}
    Suppose $2\leq d\leq n$ and let $\sQ$ be a loopless $(d,n)$--matroid. Then $\sQ$ is paving if and only if every cyclic flat of $\sQ$ has rank $0$, $d-1$, or $d$. 
\end{proposition}
\noindent So the set of proper nonempty  cyclic flats of $\sQ$ is in $\pZ^{1}(\sQ)$.

\begin{proof}[Proof of Proposition \ref{prop:pavingCyclicFlats}]
Let $\rho_{\sQ}$ be the rank function of $\sQ$. Suppose $\sQ$ is paving, and let $\eta\neq \emptyset$ be a cyclic flat.  Since $\eta$ is a union of circuits, there must be a circuit $\gamma\subset \eta$, and  $\rho_{\sQ}(\gamma)\geq d-1$ since $\sQ$ is paving.  So $\rho_{\sQ}(\eta) \geq d-1$. Conversely, suppose every cyclic flat of $\sQ$ has rank $0$, $d-1$, $d$, and let $\gamma$ be a circuit of $\sQ$ (necessarily, $\gamma\neq \emptyset$).  The closure of $\gamma$ is a cyclic flat, and hence   $\rho_{\sQ}(\gamma)\geq d-1$. This means that $\gamma$ has at least $d$ elements.
\end{proof}

The \textit{corank} vector $\sw(\sQ) \in (\R^{\binom{[n]}{d}})^{\vee}/\langle\mathbf{1}\rangle$ of $\sQ$ is
\begin{equation*}
    \sw(\sQ)_{\lambda} = d - \rho_{\sQ}(\lambda).
\end{equation*}
The goal for the rest of this subsection is to show that, when $\sQ$ is a connected and paving matroid,  the dual graph of the subdivision $\pQ(\sw(\sQ))$ is a star-shaped tree whose central node corresponds to $\sQ$. 

\begin{proposition}
\label{prop:pavingCorank}
    If $\sQ$ is a paving matroid of rank at least $2$, then the corank vector of $\sQ$ is
    \begin{align*}
        \sw(\sQ)_{\lambda} = 
        \begin{cases}
            0 & \text{ if } \lambda \in \pB(\sQ), \\
            1 & \text{ if } \lambda \in \binom{[n]}{d} \setminus \pB(\sQ).
        \end{cases}
    \end{align*}
\end{proposition}

\noindent We thank the referee for the following simplification of our original proof. 
\begin{proof}[Proof of Proposition \ref{prop:pavingCorank}]
    If $\lambda \in \binom{[n]}{d} \setminus \pB(\sQ)$, then is dependent. In particular $\rho_{\sQ}(\lambda) \leq d-1$ and $\lambda$  contains a circuit. Since $\sQ$ is a paving matroid, $\lambda$ must be a circuit, so any subset with $d-1$ elements is independent. Therefore  $\rho_{\sQ}(\lambda) = d-1$, as required. 
\end{proof}

Next, we describe the facets of the matroid polytope of a connected and paving matroid. In terms of inequalities, the hypersimplex $\Delta(d,n)$ is
\begin{equation*}
    \Delta(d,n) = \{x\in \R^{n} \, : \, 0\leq x_i \leq 1, \; x_1 + \cdots + x_n = d\}.
\end{equation*}

\begin{proposition}
\label{prop:Ld1CyclicFlat}
    Suppose $\sQ$ is a paving and connected $(d,n)$--matroid. 
    The facets of $\Delta(\sQ)$ not contained in the boundary of $\Delta(d,n)$ are given by the inequalities  
    \begin{equation}
    \label{eq:DeltaQFacets}
     \sum_{i\in \eta} x_i \leq d-1 \hspace{15pt} \text{ for } \hspace{15pt} \eta \in \pZ^1(\sQ).
    \end{equation}
\end{proposition}

\begin{proof}
Recall that a subset $\eta\subset [n]$ is nondegenerate, in the sense of \cite{GelfandSerganova}, if the restriction $\sQ_{|\eta}$ and contraction $\sQ/\eta$ are connected. 
Furthermore, the nondegenerate subsets of size $1$ and $n-1$ correspond to facets contained inside the boundary of $\Delta(d,n)$. 
By [ibid., Theorem 2], it suffices to show that the nondegenerate subsets $\eta$ of $\sQ$ such that $1<|\eta|<n-1$ are the rank $d-1$ cyclic flats $\pZ^1(\sQ)$.

Suppose $\eta\in \pZ^1(\sQ)$.  Since $\eta$ is a flat and has rank $d-1$, the contraction $\sQ/\eta$ is isomorphic to $\sU_{1, [n]\setminus \eta}$, which is connected. Now suppose $\eta'$ is a nonempty subset of $\eta$ such that $\sQ_{|\eta'}$ is a connected component of $\sQ_{|\eta}$. This means that $\eta'$ is a flat, and since $\sQ_{|\eta'}$ is connected, it has no coloops, and hence $\eta'$ is a cyclic flat of $\sQ$. By Proposition \ref{prop:pavingCyclicFlats}, $\eta' = \eta$, and so $\sQ_{|\eta}$ is connected.  

Conversely, suppose $\eta\subset [n]$ is nondegenerate, and $1<|\eta|<n-1$. Because $|\eta|<n-1$ and $\sQ/\eta$ is connected, it is loopless, and so $\eta$ is a flat. Since $\sQ_{|\eta}$ is connected, it has no coloops. So $\eta$ is a cyclic flat, and hence $\eta\in \pZ^1(\sQ)$ by Proposition \ref{prop:pavingCyclicFlats}.
\end{proof}

Given a subset $\eta \subset [n]$ of size at least $d$,  denote by $\Leaf{\eta}$ the $(d,n)$--matroid whose bases are
\begin{equation*}
    \pB(\Leaf{\eta}) = \{ \lambda \in \textstyle{\binom{[n]}{d}} \, : \, |\lambda \cap \eta| \geq d-1  \}. 
\end{equation*}
The matroid is obtained by first taking a free extension of $\sU_{d, \eta}$ by an element $a$ of $[n]\setminus \eta$, and then adding the remaining elements of $[n]\setminus \eta$ so that they are parallel to $a$.
The polytope of $\Leaf{\eta}$  is described in terms of inequalities as
\begin{equation*}
    \Delta(\Leaf{\eta}) = \{x \in \Delta(d,n) \, : \, \sum_{i\in [n]\setminus \eta} x_i \leq 1 \}.
\end{equation*}
in particular, $\Delta(\sQ)$ has only one facet not contained in the boundary of $\Delta(d,n)$.

\begin{proposition}
\label{prop:subdStar}
    Suppose $\sQ$ is a paving and connected $(d,n)$--matroid, and let $\sw$ be its corank vector. Then $\Gamma(\sw)$ is a star graph where the central node corresponds to $\sQ$ and its leaves correspond to $\Leaf{\eta}$ for $\eta \in \pZ^1(\sQ)$. 
\end{proposition}

\begin{proof}
First, we claim that the matroids of the maximal cells of $\pQ(\sw)$ are $\sQ$ and $\Leaf{\eta}$ for $\eta \in \pZ^1(\sQ)$.  Because
    \begin{equation*}
        \sQ = \sU_{\sw}^{0} 
        \hspace{20pt} \text{and} \hspace{20pt} \Leaf{\eta} = \sU_{\sw}^{\sv} \hspace{10pt}\text{(where $\sv = -\epsilon_{\eta}^{*}$ for $\eta \in \pZ^1(\sQ)$),}
    \end{equation*}
    we see that  $\sQ$ and $\Leaf{\eta}$ for $\eta \in \pZ^1(\sQ)$ are matroids of maximal cells of $\pQ(\sw)$ (they are maximal since these matroids are connected). Suppose $x \in \Delta(d,n) \setminus \Delta(\sQ)$. By the characterization of $\Delta(\sQ)$ in Formula \ref{eq:DeltaQFacets}, there is some $\eta \in \pZ^1(\sQ)$ such that 
    \begin{equation*}
        \sum_{i\in \eta} x_i > d-1, \hspace{20pt} \text{equivalently, } \hspace{20pt} \sum_{i\in [n]\setminus \eta} x_i < 1. 
    \end{equation*}
    From this, we see that $x\in \Delta(\Leaf{\eta})$, and so there are no other maximal cells of $\pQ(\sw)$. Next, we claim that the only edges of $\Gamma(\sw)$ are between $v_{\sQ}$ and $v_{\Leaf{\eta}}$. Indeed, $\Delta(\sQ)$ and $\Delta(\Leaf{\eta})$ (for $\eta \in \pZ^1(\sQ)$) intersect along a common facet whose matroid is
    \begin{equation*}
        \Leaf{\eta}' = \{\lambda \in \textstyle{\binom{[n]}{d}} \, : \, |\lambda \cap \eta| = d-1\}. 
    \end{equation*}
    Note that $\Leaf{\eta}' \cong \sU_{d-1, \eta} \oplus \sU_{1, [n]\setminus \eta}$. 
    Since $\Delta(\Leaf{\eta})$ has only one facet contained in the relative interior of $\Delta(d,n)$, these are the only edges of $\Gamma(\sw)$.      
\end{proof}

\subsection{Singular initial degenerations of the Grassmannian}
\label{sec:singularInitialDeg312}
Before describing singular initial degenerations of the open Grassmannian, we derive smoothness results on the inverse limit $\Gr(\sw)$ when $\sw$ is the corank vector of a paving and connected  matroid. 

\begin{proposition}
\label{prop:smoothSurjectiveQeta}
    For any $\eta\subset [n]$ with $d \leq |\eta| \leq n-1$, the face morphism  $\varphi_{\eta}:= \varphi_{\Leaf{\eta}, \Leaf{\eta}'}:\Gr(\Leaf{\eta};\C) \to \Gr(\Leaf{\eta}';\C)$ is smooth and surjective with connected fibers. 
\end{proposition}

\begin{proof}
For brevity, let $\sQ = \Leaf{\eta}$ and $\sQ' = \Leaf{\eta}'$.  Without loss of generality, suppose $\eta = [d+k]\setminus \{d\}$ (where $1\leq k \leq n-d$). Define index sets $\Lambda'$ and $\Lambda$ by
\begin{align*}
    &\Lambda' = \{(i,j) \, : \, i=1,2,\ldots,d-1;\, j=1,2,\ldots,k \} \cup \{(d,j) \, : \, j=k+1\ldots,n-d\}, \\
    &\Lambda = \Lambda' \cup \{(d,j) \, : \, j=1,\ldots,k\}.
\end{align*}

\noindent Let $A(x)$ be the matrix of variables
\begin{equation*}
        A(x) =\begin{bmatrix}
    1 &0 &\cdots &0 &x_{11} &\cdots &x_{1,k} & 0 & \cdots & 0 \\
    0 &1 &\cdots &0 &x_{21} &\cdots &x_{2,k} & 0 & \cdots & 0 \\
    \cdots &\cdots &\cdots &\cdots &\cdots &\cdots &\cdots & \cdots & \cdots & \cdots \\
    0 &0 &\cdots &1 &x_{d1} &\cdots &x_{d,k} & x_{d,k+1} &\cdots &x_{d,n-d}
    \end{bmatrix}.
    \end{equation*}
Then  $B_{\sQ'} = \C[x_{ij}\, : \, ij\in \Lambda']$ and $B_{\sQ} = \C[x_{ij}\, : \, ij\in \Lambda]$. The coordinate rings of $\Gr(\sQ';\C)$ and $\Gr(\sQ; \C)$ are $S_{\sQ'} = U_{\sQ'}^{-1}B_{\sQ'}$ and $S_{\sQ} = U_{\sQ}^{-1}B_{\sQ}$, respectively, and the morphism $\Gr(\sQ;\C) \to \Gr(\sQ';\C)$ is induced by the inclusion $(U_{\sQ'})^{-1}B_{\sQ}' \subset U_{\sQ}^{-1}B_{\sQ}$. In particular, these are integral affine schemes of finite-type over $\C$. To show that $\varphi_{\eta}: \Gr(\sQ;\C) \to \Gr(\sQ';\C)$ is surjective, it suffices to show that it is surjective at the level of closed points \cite[8.4.F]{VakilFOAG}. 

Let $V'$ be a $\C$-realization of $\sQ'$, i.e., $V'$ is the row span of the matrix $A(y)$ where  $y_{ij} = 0$ for $(i,j)\in \Lambda\setminus \Lambda'$ and $y_{ij}$ for $(i,j)\in \Lambda'$ are nonzero complex numbers such that $f(y_{ij}) \neq 0$ for $f\in U'$. Choose complex numbers $z_{ij}$ such that $z_{ij}$ for $(i,j)\in \Lambda\setminus \Lambda'$ are algebraically independent over $\Q$ and $z_{ij} = y_{ij}$ for $(i,j)\in\Lambda'$. Let $V$ be the row span of $A(z)$. Since the semigroup $U_{\sQ}$ are generated by finitely many polynomials with \textit{integer} coefficients, we see that $V$ is a realization of $\sQ$ and $\varphi(V) = V'$, as required. 
\end{proof}

\begin{proposition}
\label{prop:compareLimitToCenter}
Suppose $\sQ$ is a paving and connected $(d,n)$--matroid, and let $\sw$ be its corank vector. Then the morphism $\varphi_{\sQ}:\Gr(\sw) \to \Gr(\sQ;\C)$ is smooth and surjective with connected fibers. Furthermore,
    \begin{equation*}
        \dim \Gr(\sw) = \dim \Gr(\sQ;\C) + \sum_{\eta\in \pZ^1(\sQ)} (|\eta|-d+1).
    \end{equation*}
\end{proposition}

\begin{proof}
    By Proposition \ref{prop:smoothSurjectiveQeta}, the morphism
    \begin{equation*}
        \prod_{\eta\in \pZ^1(\sQ)}\varphi_{\eta}: \prod_{\eta\in \pZ^1(\sQ)} \Gr(\Leaf{\eta};\C) \to \prod_{\eta\in \pZ^1(\sQ)} \Gr(\Leaf{\eta}';\C),
    \end{equation*}
    and therefore its base-change $\varphi_{\sQ}:\Gr(\sw) \to \Gr(\sQ)$, is smooth and surjective with connected fibers. Since
    \begin{align*}
        &\dim \Gr(\Leaf{\eta};\C) = (d-1)|\eta| + n - d^2 + d - 1 \\ 
        &\dim \Gr(\Leaf{\eta}';\C) = (d-2)|\eta| + n - d^2+2d-2, 
    \end{align*}
    the proposition follows from \cite[Proposition~A.6]{CoreyGrassmannians}. 
\end{proof}

Given a closed subset $X$ of a scheme $Y$, denote by $X_{\red}$ the reduced induced scheme structure on $X$.

\begin{proposition}
\label{prop:codim0}
    If $Y$ is an $n$-dimensional finite-type reduced affine scheme, and $X\subset Y$ is a closed subscheme that is pure of dimension $n$, then $X_{\red}$ is a union of irreducible components of $Y$.
\end{proposition}

\begin{proof}
    Suppose $Y = \Spec(B)$ and $X = \Spec(B/I)$. Let $I'$ be a minimal prime of the radical of $I$. Since $\dim X = n$,  there is a length-$n$ chain of prime ideals $I' \subsetneq P_1 \subsetneq \cdots \subsetneq P_n$ in $B$. Since $\dim Y = n$, this means that $I'$ is a minimal prime of $B$, i.e., $\Spec(B/I')$ is an irreducible component of $Y$. 
\end{proof}

Let $d\geq 3$ and $n\geq d + 9$. For the remainder of this section, we denote by $\sQ$ the simple $(d,n)$--matroid obtained by taking $d-3$ free coextensions of $\sQ_{\sing}$ from \S\ref{sec:3-12}, and then $n-(d+9)$ free extensions. Denote by $\xi\subset [n]$ the subset of coextended elements. This is a connected matroid whose cyclic flats (other than $\emptyset$ and $[n]$) have rank $d-1$, and are 
\begin{equation*}
    \pZ^1(\sQ) = \{ \eta \cup \xi \, : \, \eta\in \pZ^1(\sQ_{\sing})  \}.
\end{equation*}
So $\sQ$ is paving by Proposition \ref{prop:pavingCyclicFlats}. 
We also denote by $\sw \in (\R^{\binom{n}{d}})^{\vee}$ the corank vector of $\sQ$. 
\begin{proposition}
\label{prop:singularIsom}
    The corank vector  $\sw$ lies in $\TGr^{\circ}(d,n;\C)$ and the closed immersion 
    \begin{equation*}
        \init_{\sw}\Gr^{\circ}(d,n;\C) \to \Gr(\sw)
    \end{equation*}
    is an isomorphism. 
\end{proposition}

Suppose $X\subset \G_m^{a}$ is a closed subvariety and $K=\CCt$. Given  $\sfp \in X(K)$, recall that the \textit{tropicalization} $\trop(\sfp)$ is the coordinatewise valuation of $\sfp$. The \textit{exploded tropicalization} $\ETrop(\sfp)$ is the vector whose coordinates are the coefficients of the leading terms. By \cite[Lemma~3.2]{Payne}, $\ETrop(\sfp) \in \init_{\sw} X$ where $\sw=\trop(\sfp)$.  

\begin{proof}[Proof of Proposition \ref{prop:singularIsom}]
As in the proof of Theorem \ref{thm:singular3-12}, denote by $X_1 = V(xy+x-2y)$ and $X_{2} = V(y-\omega)$, and $X_{3} = V(y-\overline{\omega})$ the three irreducible components of $\pR(\sQ_{\sing},\C)$. Let $A_{1}$, $A_{2}$ and $A_{3}$ be the matrices obtained by evaluating the matrix $A$ from \eqref{eq:representationMatrix312} at  $(x,y)=(3,-3)$, $(3, \omega)$, and $(3,\overline{\omega})$, respectively. These matrices correspond to points of $X_1\setminus (X_2\cap X_3)$,  $X_2\setminus (X_1\cap X_3)$, and $X_{3}\setminus (X_{1}\cap X_{2})$, respectively.  Define
    \begin{equation*}
    B_t =\begin{bmatrix}
    t & 0 & 2t & -t & t & 0 & -t & t & 0 & t & t & t \\
    -t & 0 & t & t & t & 0 & 3t & 0 & t & t & t & -t \\
    0 & 0 & -t & t & 0 & t & -t & -t & t & t & t & 0
    \end{bmatrix}.
    \end{equation*}
Set $C_{i} = A_{i}+B_t$  for $i=1,2,3$. Let $D'$ be a $(d-3)\times 9$ matrix, $D = [0|D']$ a $(d-3)\times 12$ matrix and $E$ a $d\times (n-d-9)$ matrix, so that the combined entries of $D'$ and $E$ are complex numbers algebraically independent over $\Q$. Finally, for $i=1,2,3$, set
\begin{equation*}
F_{i} = \left[
\begin{array}{c | c} 
  \begin{array}{c c} 
     C_{i} & 0 \\ 
     D & I_{d-3} 
  \end{array} & E \\ 
 \end{array} \right]. \hspace{20pt}
\end{equation*}
Denote by $\sfp_i(t)$ the Pl\"ucker vector of $F_{i}$ for $i=1,2,3$. Since the elements in $D$ and $E$ are algebraically independent, we have that $\val_t(\sfp_i(t)_{\lambda}) = 0$ for all $\lambda\in \binom{[n]}{d}$ such that $\lambda \not\subset \{1,2,\ldots,12\}$.  Thus, it is a routine verification to show that $\trop(\sfp_i(t)) = \sw$ for $i=1,2,3$; in particular, this implies $\sw \in \TGr^{\circ}(d,n;\C)$.  

Denote by $\varphi_{\sQ}:\Gr(\sw) \to \Gr(\sQ;\C)$ the structural morphism from the inverse limit and  $\pi:\Gr(\sw) \to \pR(\sQ_{\sing};\C)$ the composition
\begin{equation*}
    \Gr(\sw) \xrightarrow{\varphi_{\sQ}} \Gr(\sQ;\C) \to \Gr(\sQ_{\sing};\C) \to \pR(\sQ_{\sing};\C).
\end{equation*}
The morphism $\pi$ is smooth and surjective with connected fibers by Propositions \ref{prop:principalExtension}, \ref{prop:coextension} and \ref{prop:compareLimitToCenter}. 
So $\Gr(\sw)$ is a reduced affine scheme with three irreducible components. By \cite[Proposition~III.9.5]{Hartshorne}, its dimension is the fiber dimension of $\pi$, which is $d(n-d)-1$ by Proposition \ref{prop:compareLimitToCenter}, plus with $\dim \pR(\sQ_{\sing};\C)$, which is 1. So $\dim \Gr(\sw) = d(n-d)$.

Set  $\sa_i = \ETrop(\sfp_i)$ for $i=1,2,3$; these points are in $\init_{\sw}\Gr^{\circ}(d,n)$ as discussed above. Since $(\sa_i)_{\lambda} = \sfp_{i}(0)_{\lambda}$ for $i=1,2,3$ and $\lambda \in \pB(\sQ_{\sing})$, we have  $(\pi\circ\psi_{\sw})(\sa_1) \in X_1\setminus (X_2\cap X_3)$, $(\pi\circ\psi_{\sw})(\sa_2) \in X_2 \setminus (X_1\cap X_3)$ and $(\pi\circ\psi_{\sw})(\sa_3) \in X_3 \setminus (X_1\cap X_2)$ where $X_{1}$, $X_{2}$, and $X_{3}$ are the 3 irreducible components of $\pR(\sQ_{\sing};\C)$ described above. So the image of $\psi_{\sw}$ meets the 3 irreducible components of $\Gr(\sw)$ outside of their intersection. Being a flat degeneration of $\Gr^{\circ}(d,n)$, we have that the reduction of $\init_{\sw}\Gr^{\circ}(d,n)$ is pure of dimension $d(n-d)$. So we have closed immersions 
\begin{equation*}
    (\init_{\sw}\Gr^{\circ}(d,n))_{\red} \hookrightarrow  \init_{\sw}\Gr^{\circ}(d,n)\xrightarrow{\psi_{\sw}} \Gr(\sw).
\end{equation*}
The composition is an isomorphism by Proposition \ref{prop:codim0} and the fact that the image of $\psi_{\sw}$ meets all irreducible components of $\Gr(\sw)$ (outside of pairwise intersections). We conclude that $\psi_{\sw}$ is an isomorphism. 
\end{proof}

Thorem \ref{thm:not-shon-geq-12} can be formulated as follows.

\begin{theorem}
\label{thm:notSchoenPairs}
    Suppose $3 \leq d \leq \frac{n}{2}$ and $(d,n)$  is not any of the pairs in \eqref{eq:openPairs}.   Then $\Gr^{\circ}(d,n)$ and $\pR(d,n)$ are not sch\"on.  
\end{theorem}

\begin{proof}[Proof of Theorem \ref{thm:notSchoenPairs}] Let $\sQ$ and $\sw$ be as in Proposition \ref{prop:singularIsom}. The inverse limit $\Gr(\sw)$ is not smooth since there is a smooth and surjective morphism $\Gr(\sw) \to \pR(\sQ_{\sing};\C)$ and $\pR(\sQ_{\sing};\C)$ is singular by Theorem \ref{thm:singular3-12}. Also by Proposition \ref{prop:singularIsom}, the initial degeneration $\init_{\sw}\Gr^{\circ}(d,n)$ is isomorphic to $\Gr(\sw)$, and hence also singular. 
\end{proof}

\section{Concluding remarks}

We conclude with some potential directions for future work. Our examination of realization spaces of matroids did not completely exhaust available data. That is, there are still $(3,12)$--matroids that have gone unstudied, and the data set of $(4,10)$--matroids has not been examined at all. Of course, as more data becomes available, it would be natural to apply our methods to $(d,n)$--matroids for the pairs $(d,n)$ in \eqref{eq:openPairs} that we have not yet considered. 

It is also natural to ask what conditions are sufficient or necessary to conclude that the realization space of a matroid is singular. Furthermore, for a given pair $(d,n)$, what singularity types can appear?  While Mn\"ev universality guarantees all singularity types appear on \emph{some} realization space, are there obstructions that prevent certain singularity types from appearing for a particular $(d,n)$ pair? 

Finally, we revisit Question \ref{question:3-n range} from the introduction. Theorem \ref{thm:not-shon-geq-12} leaves the problem of determining whether $\Gr^{\circ}(3,n)$ is sch\"on open only for $n=9,10,11$.  Some challenges to resolving this are immediately apparent. For this range of $n$, all matroid strata are smooth by Propositions \ref{prop:3-9smooth} and \ref{prop:3-10-11smooth}, so the technique used to construct the singular initial degeneration in Proposition \ref{prop:singularIsom} fails. By \cite[Example 8.2]{CoreyGrassmannians}, for $n=9$ there exist initial degenerations such that the closed immersion of Theorem \ref{thm:closed_immersion} is not an isomorphism. Thus an exhaustive study of inverse limits over dual graphs for matroidal subdivisions of hypersimplices (as in \cite{CoreyGrassmannians} and \cite{CoreyLuber}), which is already a complex task, is insufficient to prove sch\"onness.


\end{document}